\documentclass[a4paper]{article}
\usepackage{amsmath,amsthm,amssymb,latexsym,epic}
\usepackage{graphicx,enumerate,stmaryrd}
\newtheorem{theorem}{Theorem}
\newtheorem{lemma}[theorem]{Lemma}
\newtheorem{remark}[theorem]{Remark}
\newtheorem{corollary}[theorem]{Corollary}
\newtheorem{proposition}[theorem]{Proposition}
\newtheorem{example}[theorem]{Example}

\usepackage[all]{xy}

\begin{document}
\newcommand{\mC}{\mathbb{C}}
\newcommand{\tto}{\twoheadrightarrow}
\newcommand{\mathttD}{R}
\newcommand{\la}{\lambda}
\newcommand{\ov}{\overline}
\newcommand{\mZ}{\mathbb{Z}}
\newcommand{\cO}{\mathcal{O}}
\newcommand{\HOM}{\operatorname{Hom}}
\newcommand{\END}{\operatorname{End}}

\title{Cuspidal $\mathfrak{sl}_n$-modules and deformations
of certain Brauer tree algebras}
\author{Volodymyr Mazorchuk and Catharina Stroppel}
\date{\today}

\maketitle

\begin{abstract}
We show that the algebras describing blocks of the category of cuspidal
weight (respectively generalized weight) $\mathfrak{sl}_n$-modules
are one-parameter (respectively multi-parameter) deformations of
certain Brauer tree algebras. We explicitly determine these 
deformations both graded and ungraded. The 
algebras we deform also appear as special centralizer subalgebras
of Temperley-Lieb algebras or as generalized Khovanov algebras.
They show up in the context of highest weight representations of 
the Virasoro algebra, in the context of rational representations of 
the general linear group and  Schur algebras and in the study of 
the Milnor fiber of Kleinian singularities. 
\end{abstract}

\tableofcontents

\section{Introduction and description of the results}\label{s0}

Weight modules, or, more generally, generalized weight modules, 
play an important role in the representation theory of semisimple 
complex Lie algebras. One reason for their importance is the fact 
that if we further assume the (generalized) weight spaces to be
finite dimensional, then simple modules can be classified. No doubt, 
this classification theorem (finally completed by O.~Mathieu in \cite{Mat}) 
marks a major breakthrough in the field.  Already from the results of
S.~Fernando (\cite{Fe}) and V.~Futorny (\cite{Fu}) it was known that
simple weight modules with finite dimensional weight spaces
fall into two types:
\begin{itemize}
\item the so-called {\em cuspidal}
modules, that is the ones which are not parabolically induced modules or equivalently on which all root vectors of the Lie algebra
act bijectively (\cite[Cor 1.4, Cor 1.5]{Mat}); and
\item the simple quotients of generalized Verma modules,
parabolically induced from cuspidal modules.
\end{itemize}
The second type forms the bulk of simple weight modules (and also of the 
literature on weight modules); they are easy to classify, and their structure 
and Kazhdan-Lusztig type combinatorics is now relatively well understood, see 
\cite{MS}, \cite{BFL}, \cite{Maz3} and references therein.
From \cite{Fe} (see also \cite[Prop. 1.6]{Mat}) it is known that cuspidal
modules only exist for the Lie algebras $\mathfrak{sl}_n$ (type $A$)
and  $\mathfrak{sp}_{2n}$ (type $C$), and it is the classification of
simple cuspidal modules for these two series of Lie algebras, which was
completed by Mathieu in \cite{Mat}.

The next natural step is thus to describe and understand the category 
of cuspidal (generalized) weight modules. In \cite{BKLM} it was shown 
that for the algebra $\mathfrak{sp}_{2n}$ the category of cuspidal 
weight modules is semi-simple, hence completely understood. The very 
interesting rather recent paper \cite{GS} deals with the remaining cases 
of weight modules and describes blocks of the category of cuspidal 
weight modules for the algebra $\mathfrak{sl}_{n}$ in terms of modules 
over the path algebra of a quiver with relations.

This latter paper has been the inspiration and motivation for our work, 
but we want to go much further. We first reprove the main result from  
\cite{GS} by completely different methods, which also allow us to
extend it to the category of all generalized weight modules. Moreover,
we relate the associative algebras appearing in our description
to (multi-parameter) deformations of certain self-injective symmetric algebras. 
In the case of weight modules we obtain the universal one-parameter deformation 
of the centralizer subalgebra $A^{n-1}$ for a basic projective-injective 
module in a block of parabolic category $\mathcal{O}$ for $\mathfrak{sl}_n$.

We would like to emphasize that these algebras $A^{n}$ show up in various 
other contexts, for instance as subalgebras of Temperley-Lieb algebras, 
as special examples of generalized Khovanov algebras (\cite{BS1}), as 
special instances of Brauer tree algebras (\cite{Ho1}), in the context 
of highest weight representations of the Virasoro algebra (\cite{BNW}),
in the context of rational representations of the general linear group and 
Schur algebras (\cite{Xi}), in the study of the Milnor fiber of Kleinian 
singularities (\cite{KhSe}), as a convolution algebra (\cite{SW})  etc. 
It would be interesting to know which role our deformations play in these different contexts.

Our main results are then the following: let $\mathcal{C}$,  $\hat{\mathcal{C}}$ be the category of finitely generated, cuspidal, weight (respectively generalized weight) $\mathfrak{sl}_{n}$-modules  (for $n\geq 2$ fixed).

\begin{theorem}\label{thmmain}
\begin{enumerate}[(i)]
\item\label{thmmain.1} Every non-integral or singular block of 
$\mathcal{C}$ is equivalent
to the category of finite dimensional $\mathbb{C}[[x]]$-modules.
\item\label{thmmain.2} Every non-integral or singular block of $\hat{\mathcal{C}}$
is equivalent to the category of finite dimensional
$\mathbb{C}[[x_1,x_2,\dots,x_n]]$-modules.
\item\label{thmmain.3} For $n>2$ every integral regular block of 
$\mathcal{C}$ is equivalent to the category of finite dimensional 
modules over a flat one-parameter deformation of $A^{n-1}$ which is 
non-trivial as infinitesimal deformation, where $A^{n-1}$ is the path 
algebra of the following quiver with  $n-1$ vertices
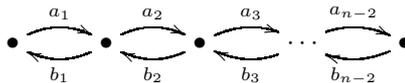
\begin{displaymath}
\xymatrix{
\bullet\ar@/^/[r]^{a_1}&\ar@/^/[l]^{b_1}\bullet
\ar@/^/[r]^{a_2}&\ar@/^/[l]^{b_2}\ar@/^/[r]^{a_3}\bullet
&\cdots\ar@/^/[l]^{b_3}\ar@/^/[r]^{a_{n-2}}
&\ar@/^/[l]^{b_{n-2}}\bullet
}
\end{displaymath}
modulo the relations $a_{i+1}a_i=0=b_{i}b_{i+1}$ and $b_ia_i=a_{i-1}b_{i-1}$
(whenever the expression makes sense) in the case $n>3$ and
$a_1b_1a_1=0=b_1a_1b_1$ in the case $n=3$.
The path length induces a non-negative $\mZ$-grading on $A^{n-1}$. Then the deformation in question is the unique (up to rescaling of the deformation
parameter) non-trivial graded one-parameter deformation of $A^{n-1}$.
This deformation is the completion of a Koszul algebra with respect 
to the graded radical.
\item\label{thmmain.4} For $n>2$ every integral regular block of 
$\hat{\mathcal{C}}$ is equivalent to the category of finite dimensional 
modules over a flat $n$-parameter deformation of $A^{n-1}$. The associative
algebra of this deformation is isomorphic to the tensor product of the
deformation described in the previous claim \eqref{thmmain.3} and 
the algebra $\mathbb{C}[[x_2,x_3,\dots,x_n]]$.
\end{enumerate}
\end{theorem}

The paper is organized as follows: Section~\ref{s1} contains 
preliminaries on cuspidal modules. In Section~\ref{s2} we
prove the first two assertions of Theorem~\ref{thmmain}, that is
we describe all singular and non-integral blocks of the category of
cuspidal (generalized) weight modules. In Section~\ref{s3} we show that
regular blocks are described by deformations of the algebra $A^{n-1}$
over the algebra describing the singular blocks. Finally, in Section~\ref{s4} 
we describe the (graded and ungraded) deformation theory of these 
associative algebras in detail. We first give a summary of basic 
results on multi-parameter deformations of associative algebras 
(Subsections~\ref{s4.1}-\ref{s4.5}, this part is completely independent 
from the Lie theory) and then use it in Subsections~\ref{s4.6} to 
complete the proof of the main theorem.
\vspace{0.5cm}

\noindent
{\bf Acknowledgments.} Most of the research was carried during visits of 
the first author to Bonn University and of the second author to Uppsala 
University. We thank the Royal Swedish Academy of Sciences, the Swedish 
Research Council and the Hausdorff Center of Mathematics in Bonn for 
financial support. We would like to thank Carl Mautner for useful 
discussions and especially Thorsten Holm for very useful 
discussions and numerous explanations.

\section{Preliminaries on cuspidal modules}\label{s1}

All algebras and categories in this paper are assumed to be
over the field $\mathbb{C}$ of complex numbers. All functors
are additive and $\mathbb{C}$-linear. All unspecified homomorphisms and
tensor products are assumed to be over $\mathbb{C}$.
We denote by $\mathbb{N}$
and $\mathbb{Z}_+$ the set of positive and nonnegative integers,
respectively.

\subsection{Weight and generalized weight
$\mathfrak{sl}_n$-modules}\label{s1.1}

Fix $n\in\{2,3,\dots\}$ and consider the complex Lie algebra $\mathfrak{gl}_n$ spanned by the matrix units $e_{i,j}$, $i,j=1,2,\dots,n$. Let $\mathfrak{g}:=\mathfrak{sl}_n$ be the Lie subalgebra of matrices with 
trace zero. We have the standard triangular decomposition $\mathfrak{g}=\mathfrak{n}_-\oplus \mathfrak{h}\oplus \mathfrak{n}_+$
of $\mathfrak{g}$, i.e. the decomposition into a sum of strictly lower 
triangular, diagonal and strictly upper triangular matrices.
Let $\Delta\subset \mathfrak{h}^*$ be the set of roots of $\mathfrak{g}$ and $\rho$ the half-sum of all positive roots. If $\alpha\in \Delta$ then $\mathfrak{g}_{\alpha}$ denotes the root space of
$\mathfrak{g}$ corresponding to $\alpha$. Let $W$ denote the Weyl group
of $\Delta$. In the following $W$will act on $\mathfrak{h}^*$ via the
{\em dot action} defined by $w\cdot \lambda=w(\lambda+\rho)-\rho$.

For $i=1,\dots,n-1$ let $\alpha_i$ be the simple (positive) root 
with the coroot $h_i=e_{i,i}-e_{i+1,i+1}\in\mathfrak{h}^*$,
and let $s_i$ be the corresponding  simple reflection in $W$. For $\lambda\in \mathfrak{h}^*$ and $i\in\{1,2,\dots,n-1\}$ set
$\lambda_{i}=\lambda(h_i)$.  An element $\lambda\in \mathfrak{h}^*$ is
called a {\it weight}. It is {\em integral} provided that all $\lambda_{i}$'s
are integers. It is {\it singular}, if the cardinality of the stabilizer
$W_\lambda$ of $\lambda$ is at least two; it is {\it regular} otherwise.

A $\mathfrak{g}$-module $M$ is called a {\em weight module} or a
{\em generalized weight module} if
$M=\oplus_{\lambda\in\mathfrak{h}^*}M_{\lambda}$ or
$M=\oplus_{\lambda\in\mathfrak{h}^*}M^{\lambda}$, respectively, where
\begin{gather*}
M_{\lambda}:=\{v\in M\mid hv=\lambda(h)v\text{ for all }h\in\mathfrak{h}\},\\
M^{\lambda}:=\{v\in M\mid(h-\lambda(h))^kv=0\text{ for all }h\in\mathfrak{h}
\text{ and }k\gg 0\}.
\end{gather*}
Here $\lambda$ is called a {\em weight of $M$} with {\em weight space} $M_{\lambda}$ and {\em generalized weight space} $M^{\lambda}$. Obviously, every weight module is a generalized
weight module, whereas the converse need not to be true. However, from the defining relations for $\mathfrak{g}$ it follows that every simple generalized  weight module is in fact a weight module.
Note that $M_{\lambda}\subset M^{\lambda}$, moreover,
$M_{\lambda}\neq 0$ if and only if $M^{\lambda}\neq 0$.
The set $\mathrm{supp}(M):=\{\lambda\in
\mathfrak{h}^*\mid M_{\lambda}\neq 0\}$ is called the {\em support} of
$M$.  A weight module $M$ is called {\em pointed} provided that
$\dim M_{\lambda}= 1$ for some $\lambda$ (for indecomposable cuspidal 
modules the latter is equivalent to the requirement $\dim M_{\lambda}\leq 
1$ for all $\lambda$).  In the following we want to
assume (following \cite{Mat}) that all (generalized) weight spaces are 
finite dimensional. For examples of simple weight modules with
infinite dimensional weight spaces we refer to \cite{DFO}.
\begin{center}
\vspace{2mm}
\it{From now on: all generalized weight spaces are assumed to be finite dimensional.}
\vspace{2mm}
\end{center}
If $M$ is a simple weight $\mathfrak{g}$-module, then there exists
$\lambda\in \mathfrak{h}^*$ such that
$\mathrm{supp}(M)\subset \lambda+\mathbb{Z}\Delta$ (in fact, for
the latter it is enough to assume that $M$ is indecomposable). If
$\mathrm{supp}(M)=\lambda+\mathbb{Z}\Delta$, the $M$ is called
a {\em dense} module. By \cite{DMP}, if $M$ is a simple dense module,
then each nonzero element of any $\mathfrak{g}_{\alpha}$, $\alpha\in\Delta$,
acts bijectively on $M$. Hence any simple dense module is cuspidal (in the sense as defined in the introduction). Obviously, every simple cuspidal $\mathfrak{g}$-module is dense.

Both categories $\mathcal{C}$ and $\hat{\mathcal{C}}$, defined in the
introduction, are abelian categories and $\mathcal{C}$ is a full 
subcategory of $\hat{\mathcal{C}}$. Moreover, $\mathcal{C}$ and 
$\hat{\mathcal{C}}$ have the same simple objects, namely simple cuspidal 
modules, which we are going to describe in the next subsection.

\subsection{Classification of simple cuspidal (dense) $\mathfrak{sl}_n$-modules}\label{s1.2}

Let $U=U(\mathfrak{g})$ denote the universal enveloping algebra 
of $\mathfrak{g}$ with center $Z$. For $\lambda\in\mathfrak{h}^*$ 
let $L(\lambda)$ be the simple highest weight module with highest 
weight $\lambda$, and $\chi_{\lambda}:Z\to\mathbb{C}$ the algebra
homomorphism (i.e. the {\em central character}), which defines the
(scalar) action of $Z$ on $L(\lambda)$ (see e.g. \cite[Chapter~7]{Di}). 
Denote by $C_{\operatorname{dom}}$ the closure of the dominant Weyl
chamber with respect to the dot-action. Then there is a natural
bijection between elements in $C_{\operatorname{dom}}$ and algebra 
homomorphisms $\chi:Z\to\mathbb{C}$, given by 
$C_{\operatorname{dom}}\ni\lambda\mapsto \chi_{\lambda}$. 
The action of $Z$ preserves all (generalized) weight spaces and hence 
is locally finite on all modules $M\in\hat{\mathcal{C}}$. This gives a 
natural decomposition
\begin{displaymath}
\hat{\mathcal{C}}=\displaystyle{\bigoplus_{\lambda\in C_{\operatorname{dom}}}\hat{\mathcal{C}}_{\lambda}}, 
\end{displaymath}
where $\hat{\mathcal{C}}_{\lambda}$ denotes the full subcategory of
$\hat{\mathcal{C}}$, consisting of all modules $M$ such that
$(z-\chi_{\lambda}(z))^km=0$ for all $m\in M$, $z\in Z$ and all $k\gg 0$. 
This decomposition induces the decomposition
$\mathcal{C}=\oplus_{\lambda\in C_{\operatorname{dom}}}\mathcal{C}_{\lambda}$, 
where
$\mathcal{C}_{\lambda}=\mathcal{C}\cap \hat{\mathcal{C}}_{\lambda}$.
By \cite[Section~8]{Mat}, the category $\hat{\mathcal{C}}_{\lambda}$ 
is nonzero in precisely the following cases:

\begin{enumerate}[(I)]
\item\label{case1} $\lambda$ is regular, non-integral, and the set
$\{j\mid \lambda_j\not\in\mathbb{Z}_+\}$ coincides either with $\{1\}$
or $\{n-1\}$ or with $\{i,i+1\}$ for some $i\in\{1,2,\dots,n-2\}$, moreover
$\la_i+\la_{i+1}\in\{-1,0,1,2,\dots\}$ in the latter case;
\item\label{case2} $\lambda$ is integral and singular and its stabilizer
with respect to the dot action is $\langle s_i\rangle$ for some
$i\in\{1,2,\dots,n-1\}$;
\item\label{case3} $\lambda$ is regular and integral.
\end{enumerate}

To describe simple cuspidal modules in each case we need an additional 
tool from \cite{Mat}. Let $\mathttD\subset\Delta$ be a nonempty set 
of roots such that
\begin{displaymath}
\alpha,\beta\in \mathttD\Rightarrow\alpha+\beta\not\in \Delta\cup \{0\}.
\end{displaymath}
For every $\alpha\in \mathttD$ we fix a non-zero root
vector $X_{\alpha}\in\mathfrak{g}_{\alpha}$.
Note that $X_{\alpha}X_{\beta}=X_{\beta}X_{\alpha}$ for all
$\alpha,\beta\in \mathttD$. Consider the multiplicative
subset $S(\mathttD)$ of $U$, consisting of all
elements of the form $\prod_{\alpha\in \mathttD}X_{\alpha}^{m_{\alpha}}$,
where all $m_{\alpha}\in \mathbb{Z}_+$. The set $S(\mathttD)$ is an Ore
subset of $U$ and hence we have the corresponding Ore localization
$U_{S(\mathttD)}$ (note that the localization is independent, up to 
isomorphism, of the
choice of the root vectors). By \cite[Lemma 4.3]{Mat} there is a unique
polynomial map on $\mC^{|\mathttD|}$ with values in algebra automorphisms of
$U_{S(\mathttD)}$, $\mathbf{x}=(x_{\alpha})_{\alpha\in\mathttD} \mapsto \Phi^{\mathttD}_{\mathbf{x}}$,
such that in case $\mathbf{x}$ is a tuple of integers we have
\begin{displaymath}
\Phi^{\mathttD}_{\mathbf{x}}(u)=
\prod_{\alpha\in \mathttD}X_{\alpha}^{-x_{\alpha}}
\cdot u \cdot \prod_{\alpha\in \mathttD}X_{\alpha}^{x_{\alpha}}\quad
\text{ for all }\quad u\in U_{S(\mathttD)}.
\end{displaymath}
If $\varphi$ is an automorphism of some Ring $\mathcal{R}$ and $M$
an $\mathcal{R}$-module we denote by $M^\varphi$ the module $M$ with
the $\mathcal{R}$-module action twisted by $\varphi$. Note that every
cuspidal module is automatically a $U_{S(\mathttD)}$-module. Now the
main  classification result of \cite{Mat} is the following:

\begin{theorem}[Classification theorem]  \label{thm1}
In all cases \eqref{case1}-\eqref{case3}, any simple module in
$\hat{\mathcal{C}}_{\lambda}$ is isomorphic to a module of the form
$\big( U_{S(\mathttD)}\otimes_U L(\mu)\big)^{\Phi^{\mathttD}_{\mathbf{x}}}$ 
for appropriate $\mu$, $\mathbf{x}$, $\mathttD$.
In cases \eqref{case1} and \eqref{case2} it is $\mu=\la$, whereas 
in case \eqref{case3} we have $\mu=w\cdot\la$ with  $w\in\{s_1,s_1s_2,\dots,s_{1}s_{2}\cdots s_{n-2}s_{n-1}\}=:W^{\operatorname{short}}$, the set of shortest coset representatives in $S_{n-1}\backslash S_{n}$ different from the identity.
\end{theorem}

In each case the category $\hat{\mathcal{C}}_{\lambda}$
has infinitely many isomorphism classes of simple objects
(as the module $L(\mu)^{\Phi^{\mathttD}_{\mathbf{x}}}$ is generically
simple and there are infinitely many $\mathbf{x}$). To get rid of this 
problem one can decompose
further: for $\xi\in \mathfrak{h}^*/\mathbb{Z}\Delta$ let
$\hat{\mathcal{C}}_{\lambda,\xi}$ denote the full
subcategory of $\hat{\mathcal{C}}_{\lambda}$, which consists of all
$M$ such that $\mathrm{supp}(M)=\xi$. Then we have the
decomposition
$\hat{\mathcal{C}}_{\lambda}=\oplus_{\xi}\hat{\mathcal{C}}_{\lambda,\xi}$
and the induced decomposition
$\mathcal{C}_{\lambda}=\oplus_{\xi}\mathcal{C}_{\lambda,\xi}$.
From \cite[Section 11]{Mat} and Theorem~\ref{thm1}
one obtains the following:

\begin{corollary}\label{cor2}
\begin{enumerate}[(i)]
\item In cases \eqref{case1} and \eqref{case2}, the category $\hat{\mathcal{C}}_{\lambda,\xi}$ has at most one isomorphism class
of simple modules.
\item\label{cor2.3} In case \eqref{case3} every nonzero
$\hat{\mathcal{C}}_{\lambda,\xi}$ (and $\mathcal{C}_{\lambda,\xi}$)
is indecomposable and has exactly $n-1$ isomorphism classes of
simple modules, naturally in bijection with $W^{\operatorname{short}}$.
\end{enumerate}
\end{corollary}

\subsection{Reduction to special blocks}\label{s1.3}

As shown in \cite[Section11]{Mat}, Corollary~\ref{cor2} describes 
and indexes {\em blocks} (i.e. indecomposable
direct summands) of the categories $\hat{\mathcal{C}}$ and $\mathcal{C}$.
The goal of this paper is to describe associative algebras, whose
categories of finite dimensional modules are equivalent to
$\hat{\mathcal{C}}_{\lambda,\xi}$ or $\mathcal{C}_{\lambda,\xi}$.
In this subsection we reduce this problem to some very special blocks.
We would like to start with the following easy but very important
observation, which gets rid of the parameter $\xi$:

\begin{proposition}\label{prop3}
Assume we are in one of the cases \eqref{case1}--\eqref{case3} and $\xi_1,\xi_2\in \mathfrak{h}^*/\mathbb{Z}\Delta$.
If $\hat{\mathcal{C}}_{\lambda,\xi_k}$ is nonzero for $k=1,2$ then we have equivalences of categories
\begin{eqnarray*}
\hat{\mathcal{C}}_{\lambda,\xi_1}\cong
\hat{\mathcal{C}}_{\lambda,\xi_2},&&\mathcal{C}_{\lambda,\xi_1}\cong
\mathcal{C}_{\lambda,\xi_2}.
\end{eqnarray*}
\end{proposition}

\begin{proof}
Let $\mathttD$ be as given by Theorem~\ref{thm1}, then $R$ 
generates $\mathfrak{h}^*$. Choose arbitrary $\nu_i\in\xi_i$, $i=1,2$.
From the polynomiality
of $\Phi_{\mathbf{x}}^{\mathttD}$ it follows that there exists
$\mathbf{x}$ such that every $h\in\mathfrak{h}$ is mapped by
$\Phi_{\mathbf{x}}^{\mathttD}$ to $h+(\nu_2-\nu_1)(h)$. Indeed, one 
directly verifies the formula 
\begin{displaymath}
\Phi_{\mathbf{x}}^{\mathttD}(h)=
\prod_{\alpha\in R}X_\alpha^{-x_{\alpha}}\, h\,
\prod_{\alpha\in R}X_{\alpha}^{x_{\alpha}}=
h+\sum_{\alpha\in R}x_{\alpha}\alpha(h) 
\end{displaymath}
and then chooses $\mathbf{x}$ accordingly (which is possible as
$R$ generates $\mathfrak{h}^*$). 
Such $\Phi_{\mathbf{x}}^{\mathttD}$ thus maps
$\hat{\mathcal{C}}_{\lambda,\xi_1}$ to $\hat{\mathcal{C}}_{\lambda,\xi_2}$ 
and is an equivalence with inverse $\Phi_{-\mathbf{x}}^{\mathttD}$.
As $\Phi_{\mathbf{x}}^{\mathttD}$ maps linear polynomials over
$\mathfrak{h}$ to linear polynomials, it restricts to an
equivalence from $\mathcal{C}_{\lambda,\xi_1}$ to
$\mathcal{C}_{\lambda,\xi_2}$.
\end{proof}

Further reduction is given by the following:

\begin{proposition}\label{prop4}
Suppose we are given $\lambda_k$, $k=1,2$, in the situation of either 
of the cases \eqref{case1}-\eqref{case3}. If $\lambda_1-\lambda_2$ 
is integral,  then there are equivalences of  categories
\begin{eqnarray*}
\hat{\mathcal{C}}_{\lambda_1}\cong\hat{\mathcal{C}}_{\la_2},
&&\mathcal{C}_{\lambda_1}\cong\mathcal{C}_{\la_2}.
\end{eqnarray*}
\end{proposition}

\begin{proof}
Assume that we are in case \eqref{case3} or in case \eqref{case1}. 
Then the desired equivalence is well-known: it is given by the 
projective functor (see \cite{BG}) (=translation functor in the sense 
of \cite{Ja}) $\mathrm{T}_{\lambda_1}^{\la_2}$, viewed as a functor 
from $\mathcal{C}_{\la_1}$ to $\mathcal{C}_{\la_2}$ or as a functor 
from $\hat{\mathcal{C}}_{\la_1}$ to $\hat{\mathcal{C}}_{\la_2}$,
respectively. The same argument applies in case \eqref{case2} 
provided that $\lambda_1$ has the same stabilizer as $\la_2$.
In case the stabilizers are different, we may assume that they are 
generated by $s_{j_1}$ and $s_{j_2}$ for some $j_1$ and $j_2$ from
$\{1,2,\dots,n-1\}$. In case $j_2=j_1\pm 1$, the equivalence is given 
by projective functors as in \cite[5.9]{Ja} (see Subsection~\ref{s3.3} 
for more details). The general case follows by composing $|j_2-j_1|$ 
such equivalences.
\end{proof}

\begin{corollary}\label{cor5}
Assume we are given $\mathcal{C}_{\lambda_1}$ of type as in case 
\eqref{case1} or \eqref{case2}. Then one can find $\mathcal{C}_{\lambda_2}$ 
of the same type such that
\begin{enumerate}[(a)]
\item $\hat{\mathcal{C}}_{\lambda_1}\cong\hat{\mathcal{C}}_{\la_2}$ and $\mathcal{C}_{\lambda_1}\cong\mathcal{C}_{\la_2}$, and
\item all simple modules in $\hat{\mathcal{C}}_{\la_2}$ are pointed.
\end{enumerate}
\end{corollary}

\begin{proof}
By \cite[Theorem~4.2]{BL2}, every simple cuspidal $\mathfrak{g}$-module
is a submodule of a tensor product of a pointed cuspidal module and a
finite dimensional module. This implies the existence of some $\la_2$
such that simple modules in ${\mathcal{C}}_{\lambda_2}$ are pointed and
the difference $\lambda_1-\la_2$ is integral. Then the claim follows from
Proposition~\ref{prop4}.
\end{proof}

\subsection{Realization of pointed cuspidal modules}\label{s1.4}

Here we present a realization of pointed cuspidal modules from \cite{BL1}.
For $\mathbf{a}=(a_1,a_2,\dots,a_n)\in\mathbb{C}^n$ consider the
complex vector space $N(\mathbf{a})$ with the formal basis
\begin{displaymath}
\{\mathbf{v}_{\mathbf{b}}=x_1^{a_1+b_1}x_2^{a_2+b_2}\cdots x_n^{a_n+b_n}\mid
b_i\in\mathbb{Z} \text{ for all }i, \text{ and }
b_1+b_2+\dots+b_n=0\}.
\end{displaymath}
Then $N(\mathbf{a})$ can be turned into a $\mathfrak{g}$-module by 
restricting the natural $\mathfrak{gl}_n$-action in which the matrix 
unit $e_{i,j}$ acts as the differential operator
$x_i\frac{\partial}{\partial x_j}$. More precisely, if $\{\varepsilon_{i}:i=1,2,\dots,n\}$ denotes
the standard basis of $\mathbb{C}^n$, then clearly $\frac{\partial}{\partial x_j}\mathbf{v}_{\mathbf{b}}= (a_j+b_j)
\mathbf{v}_{\mathbf{b}-\varepsilon_{j}}$, and
\begin{eqnarray}\label{eq1}
e_{i,j}\cdot \mathbf{v}_{\mathbf{b}}= (a_j+b_j)
\mathbf{v}_{\mathbf{b}+\varepsilon_{i}-\varepsilon_{j}}.
\end{eqnarray}
These modules exhaust the simple pointed cuspidal ones:

\begin{theorem}[\cite{BL1}] \label{thm5}
Every simple pointed cuspidal $\mathfrak{g}$-module is isomorphic
to $N(\mathbf{a})$ for some $\mathbf{a}$ as above.
\end{theorem}
Note that \eqref{eq1} defines a cuspidal module if and only if
$a_i\not\in\mathbb{Z}$ for all $i$.

\subsection{Connection to associative algebras}\label{s1.5}

The categories $\hat{\mathcal{C}}_{\lambda,\xi}$ and 
$\mathcal{C}_{\lambda,\xi}$ are length categories containing 
finitely many isomorphism classes of simple objects. 
Hence by abstract nonsense (see for example
\cite[Section~7]{Ga}), these categories are equivalent
to categories of finite length modules over some complete associative
algebras $\hat{D}_{\lambda,\xi}$ and $D_{\lambda,\xi}$, respectively.
The latter algebra is a quotient of the former, since $\mathcal{C}_{\lambda,\xi}$
is a subcategory of $\hat{\mathcal{C}}_{\lambda,\xi}$.
We will show that these categories are Ext-finite in the sense that $\operatorname{Ext}^1(M,N)$ is finite dimensional for any objects $M$ and $N$. For $k\in\mathbb{N}$ we denote by $\hat{\mathcal{C}}^k_{\lambda,\xi}$
the full subcategory of $\hat{\mathcal{C}}_{\lambda,\xi}$ given by all objects with Loewy lengths at most $k$. Analogously we define $\mathcal{C}^k_{\lambda,\xi}$. These categories contain then enough 
projectives and so are equivalent to module categories over some 
finite dimensional algebras (the opposite of the endomorphism algebras 
of a minimal projective generator).
The algebras $\hat{D}_{\lambda,\xi}$ and $D_{\lambda,\xi}$ are then limits of
these finite-dimensional algebras corresponding to the directed systems
given by the natural inclusions $\hat{\mathcal{C}}^k_{\lambda,\xi}\subset
\hat{\mathcal{C}}^{k+1}_{\lambda,\xi}$ and $\mathcal{C}^k_{\lambda,\xi}\subset
\mathcal{C}^{k+1}_{\lambda,\xi}$, respectively.

\subsection{Gelfand-Zetlin realization}\label{s1.6}

In this subsection we recall the realization of simple cuspidal modules
from \cite{Maz}, which uses the Gelfand-Zetlin approach.
For $i\in\{2,3,\dots\}$ set $\mathfrak{g}_i=\mathfrak{sl}_i$,
$U_i=U(\mathfrak{g}_i)$ and let $Z_i$ denote the center
of $U_i$. We consider $\mathfrak{g}_i$ as a subalgebra
of $\mathfrak{g}_{i+1}$ with respect to the embedding into the
left upper corner. Let $\Gamma=\Gamma_n$ denote the subalgebra of
$U=U_n$, generated by $\mathfrak{h}$
and all $Z_i$, $i\leq n$. The algebra $\Gamma$ is called the
{\em Gelfand-Zetlin} subalgebra of $U$ (\cite{DFO}). The algebra $\Gamma$ 
is a maximal commutative subalgebra of $U$ and is 
isomorphic to the polynomial algebra in $\frac{n(n+1)}{2}-1$ variables
(as generators one could use any basis of $\mathfrak{h}$ and any set of
generators for each $Z_i$, which, in turn, is a polynomial algebra in
$i-1$ variables). By \cite[Lemma~3.2]{Ov}, $U$ is free both as left and
as right $\Gamma$-module.

For a $\mathfrak{g}$-module $M$ and a homomorphism
$\chi:\Gamma\to\mathbb{C}$ set
\begin{displaymath}
M_{\chi}=\{v\in M\vert (g-\chi(g))^kv=0\text{ for all }g\in\Gamma
\text{ and }k\gg 0\}.
\end{displaymath}
The module $M$ is called a {\em Gelfand-Zetlin} module provided
that $M=\oplus_{\chi}M_{\chi}$ and all $M_{\chi}$ are finite dimensional.
As $\Gamma$ is commutative and contains $\mathfrak{h}$,
any generalized weight module with finite-dimensional weight spaces,
in particular, any module from $\hat{\mathcal{C}}$,
is obviously a Gelfand-Zetlin module. Conversely, 
any  Gelfand-Zetlin  module is weight, but with
infinite dimensional weight spaces in general (see \cite{DFO} 
for details).

A character $\chi:\Gamma\to\mathbb{C}$ is usually described by the
corresponding {\em Gelfand-Zetlin tableau}, that means an equivalence 
class of patterns of the form
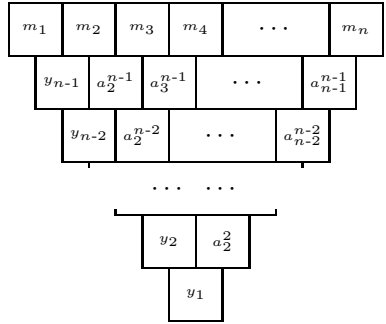
\begin{equation}\label{gztableau}
\begin{picture}(160.00,140.00)
\drawline(10.00,130.00)(150.00,130.00)
\drawline(10.00,110.00)(150.00,110.00)
\drawline(20.00,90.00)(140.00,90.00)
\drawline(30.00,70.00)(130.00,70.00)
\drawline(50.00,50.00)(110.00,50.00)
\drawline(60.00,30.00)(100.00,30.00)
\drawline(70.00,10.00)(90.00,10.00)
\drawline(10.00,130.00)(10.00,110.00)
\drawline(30.00,130.00)(30.00,110.00)
\drawline(50.00,130.00)(50.00,110.00)
\drawline(70.00,130.00)(70.00,110.00)
\drawline(90.00,130.00)(90.00,110.00)
\drawline(150.00,130.00)(150.00,110.00)
\drawline(130.00,130.00)(130.00,110.00)
\drawline(20.00,90.00)(20.00,110.00)
\drawline(40.00,90.00)(40.00,110.00)
\drawline(60.00,90.00)(60.00,110.00)
\drawline(80.00,90.00)(80.00,110.00)
\drawline(120.00,90.00)(120.00,110.00)
\drawline(140.00,90.00)(140.00,110.00)
\drawline(30.00,70.00)(30.00,90.00)
\drawline(50.00,70.00)(50.00,90.00)
\drawline(70.00,70.00)(70.00,90.00)
\drawline(130.00,70.00)(130.00,90.00)
\drawline(110.00,70.00)(110.00,90.00)
\drawline(70.00,10.00)(70.00,30.00)
\drawline(90.00,10.00)(90.00,30.00)
\drawline(60.00,50.00)(60.00,30.00)
\drawline(80.00,50.00)(80.00,30.00)
\drawline(100.00,50.00)(100.00,30.00)
\drawline(50.00,50.00)(50.00,52.00)
\drawline(110.00,50.00)(110.00,52.00)
\drawline(120.00,70.00)(120.00,68.00)
\drawline(40.00,70.00)(40.00,68.00)
\put(70.00,60.00){\makebox(0,0)[cc]{$\cdots$}}
\put(90.00,60.00){\makebox(0,0)[cc]{$\cdots$}}
\put(110.00,120.00){\makebox(0,0)[cc]{$\cdots$}}
\put(100.00,100.00){\makebox(0,0)[cc]{$\cdots$}}
\put(90.00,80.00){\makebox(0,0)[cc]{$\cdots$}}
\put(80.00,20.00){\makebox(0,0)[cc]{{\tiny $y_{1}$}}}
\put(70.00,40.00){\makebox(0,0)[cc]{{\tiny $y_{2}$}}}
\put(90.00,40.00){\makebox(0,0)[cc]{{\tiny $a_{2}^{2}$}}}
\put(20.00,120.00){\makebox(0,0)[cc]{{\tiny $m_{1}$}}}
\put(40.00,120.00){\makebox(0,0)[cc]{{\tiny $m_{2}$}}}
\put(60.00,120.00){\makebox(0,0)[cc]{{\tiny $m_{3}$}}}
\put(80.00,120.00){\makebox(0,0)[cc]{{\tiny $m_{4}$}}}
\put(140.00,120.00){\makebox(0,0)[cc]{{\tiny $m_{n}$}}}
\put(30.00,100.00){\makebox(0,0)[cc]{{\tiny $y_{n\text{-}1}$}}}
\put(40.00,80.00){\makebox(0,0)[cc]{{\tiny $y_{n\text{-}2}$}}}
\put(50.00,100.00){\makebox(0,0)[cc]{{\tiny $a^{n\text{-}1}_{2}$}}}
\put(70.00,100.00){\makebox(0,0)[cc]{{\tiny $a^{n\text{-}1}_{3}$}}}
\put(130.00,100.00){\makebox(0,0)[cc]{{\tiny $a^{n\text{-}1}_{n\text{-}1}$}}}
\put(60.00,80.00){\makebox(0,0)[cc]{{\tiny $a^{n\text{-}2}_{2}$}}}
\put(120.00,80.00){\makebox(0,0)[cc]{{\tiny $a^{n\text{-}2}_{n\text{-}2}$}}}
\end{picture} 
\end{equation}
where entries are complex numbers. Two patterns are in the same 
equivalence class if they differ by a permutation which only permutes 
entries inside rows (hence keeps the multiset of entries in each row fixed). 
Given such a tableau, the 
corresponding values of $\chi$ on standard generators of $\Gamma$ are computed 
as certain shifted symmetric functions in the entries of the tableau, 
see  \cite{DFO} for details. We will not need these explicit formulae.
The number of entries in a tableau exceeds the number of generators of
$\Gamma$ by one. This is due to the fact that the Gelfand-Zetlin
combinatorics is usually used for the Lie algebra $\mathfrak{gl}_n$
(insted of $\mathfrak{sl}_n$) where we have an extra central element.

One of the main advantages of Gelfand-Zetlin modules is that many 
Gelfand-Zetlin modules admit a so-called {\em tableau realization},
that is an explicit combinatorial construction in which a basis of the
module is indexed by a set of tableaux (defined by imposing some conditions
on the entries) and the action of the generators $e_{i,i+1}$ and
$e_{i+1,i}$ of $\mathfrak{g}$ is given by the so-called
{\em Gelfand-Zetlin formulae}, see \cite{DFO}, \cite{Maz0}, 
\cite{Maz} for details. In such a realization the tableaux 
represents a basis of common eigenvectors for $\Gamma$
(as mentioned above, the corresponding eigenvalues are computed as
certain symmetric functions), and the action of $e_{i,i+1}$ 
(resp. $e_{i+1,i}$) maps a basis vector corresponding to a tableau $t$ 
to a linear combination (with some coefficients) 
of basis vectors corresponding to 
tableaux obtained from $t$ by adding (resp. subtracting) the complex number 
$1$ to one of the entries in the $i$-th row from the bottom. The 
coefficients can be expressed as certain rational functions in the entries of 
the $i$-th and $(i+1)$-st (resp. the $i$-th and $(i-1)$-st) rows. 
Again, we will not need these explicit formulae and refer to \cite{DFO} 
and \cite{Maz} for details. 
The Gelfand-Zetlin formulae have denominators in which all
possible differences between entries in the same row (up to row $n-1$)
occur. Hence a necessary condition for the existence of a tableau 
realization is that we only use tableaux without multiple entries
in the rows $2$ to $n-1$.

Examples of modules that have such a realization 
include almost all simple cuspidal modules.
Choose arbitrary $\mathbf{m}=(m_1,m_2,\dots,m_n)\in\mathbb{C}^n$ 
and $\mathbf{x}=(x_1,x_2,\dots,x_{n-1})\in\mathbb{C}^{n-1}$
such that $m_i-m_{i+1}\in\mathbb{N}$ for all $i>1$, 
$x_i-x_{i+1}\not\in\mathbb{Z}$ for all $i$,
$x_{n-1}-m_1\not\in\mathbb{Z}$ and  $x_{i}-m_2\not\in\mathbb{Z}$ 
for all $i$. Let ${T}(\mathbf{m},\mathbf{x})$ consist 
of all tableaux of the form \eqref{gztableau} such that the 
following conditions are satisfied: $y_i-x_i\in\mathbb{Z}$, 
$a^{i}_{j}-a^{i-1}_{j}\in\mathbb{Z}_+$, 
$a^{i-1}_{j}-a^{i}_{j+1}\in \mathbb{N}$, 
$m_i-a^{n-1}_{i}\in\mathbb{Z}_+$
and $a^{n-1}_{i}-m_{i+1}\in \mathbb{N}$
whenever the expression makes sense. Let 
$M(\mathbf{m},\mathbf{x})$ denote the vector space with 
basis $T(\mathbf{m},\mathbf{x})$. Then Gelfand-Zetlin formulae
define on $M(\mathbf{m},\mathbf{x})$ the structure of a cuspidal
$\mathfrak{g}$-module (\cite{Maz}). The vector
$\mathbf{m}$ corresponds to $\lambda$. If $\lambda$ is non-integral
(i.e. $m_1-m_2\not\in\mathbb{Z}$) or if $\lambda$ is singular
(i.e. $m_1=m_s$ for some $s>1$), then the module
$M(\mathbf{m},\mathbf{x})$ is simple. It is pinned provided that
$m_i=m_{i+1}+1$ for all $i>1$.

For integral regular $\lambda$ (i.e. $m_1-m_2\in\mathbb{Z}$,
$m_1\neq m_s$ for all $s>1$) the situation is slightly more complicated.
In this case the module $M(\mathbf{m},\mathbf{x})$ is always
indecomposable. It is simple if only if
$m_1-m_2\in\mathbb{N}$ or $m_n-m_1\in \mathbb{N}$. In all other cases
this module has length two. Its unique proper submodule is the linear
span of all tableaux satisfying the additional condition as follows:
Let $s\in\{2,3,\dots,n-1\}$ be such that $m_s-m_1\in\mathbb{N}$
and $m_1-m_{s+1}\in \mathbb{N}$. Set $k_1=m_s$,
$k_2=m_2$, $k_3=m_3$,\dots, $k_{s-1}=m_{s-1}$, 
$k_s=m_1$, $k_{s+1}=m_{s+1}$,\dots, $k_n=m_n$. 
In this notation the additional condition reads: 
$k_i-a^{n-1}_{i}\in\mathbb{Z}_+$ and 
$a^{n-1}_{i}-k_{i+1}\in \mathbb{N}$ whenever the expression 
makes sense. The quotient has a basis given by all other tableaux.

\section{Blocks in cases \eqref{case1} and \eqref{case2}}\label{s2}

Our main result in this section is the following statement:

\begin{theorem}\label{thm6}
Let $\lambda$ be as in case \eqref{case1} or \eqref{case2} and
$\xi$ be such that $\hat{\mathcal{C}}_{\lambda,\xi}$ is nonzero.
\begin{enumerate}[(i)]
\item\label{thm6.1} The category $\hat{\mathcal{C}}_{\lambda,\xi}$
is equivalent to the category of finite dimensional
$\mathbb{C}[[x_1,x_2,\dots,x_n]]$-modules.
\item\label{thm6.2} The category $\mathcal{C}_{\lambda,\xi}$
is equivalent to the category of finite dimensional
$\mathbb{C}[[x]]$-modules.
\end{enumerate}
\end{theorem}

Theorem~\ref{thm6} says that $\hat{D}_{\lambda,\xi}\cong
\mathbb{C}[[x_1,x_2,\dots,x_n]]$ and $D_{\lambda,\xi}\cong
\mathbb{C}[[x]]$. The second statement of Theorem~\ref{thm6}
is contained in \cite[Section~5]{GS}. For $n=2$ both statements
are true for all $\lambda$ (even for case \eqref{case3}). This 
special case is at least partly an unpublished result of P. Gabriel
(see \cite[7.8.16]{Di}), and is completely contained in \cite{Dr} (see e.g. \cite[Chapter~3]{Maz2} for detailed proofs).

The categories $\hat{\mathcal{C}}_{\lambda,\xi}$ (and then also
$\mathcal{C}_{\lambda,\xi}$) appearing in Theorem \ref{thm6} contain a unique (up to isomorphism)
simple object thanks to Corollary~\ref{cor2}. By Corollary~\ref{cor5}, we may assume (for the rest of this section) this
simple module to be pointed, and so isomorphic to
$N(\mathbf{a})$ for some $\mathbf{a}\in \mathbb{C}^n$, as described
in Subsection~\ref{s1.4}. We are going to prove claim \eqref{thm6.1}
and then deduce claim \eqref{thm6.2} by restriction. We first need
some preparation, the actual proof will then finally appear in
Subsection~\ref{s2.4}.

\begin{corollary}\label{cor16}
\begin{enumerate}[(i)]
\item\label{cor16.1} 
Let $L$ is a simple object in $\hat{\mathcal{C}}^n_{\lambda,\xi}$ then $\dim \mathrm{Ext}^1_{\mathfrak{g}}(L,L)=n$, in particular, this
space is finite dimensional.
\item\label{cor16.2}
For any $M,N\in \hat{\mathcal{C}}_{\lambda,\xi}$ we have
$\mathrm{Ext}^1_{\mathfrak{g}}(M,N)<\infty $.
\end{enumerate}
\end{corollary}

\begin{proof}
As $\hat{\mathcal{C}}_{\lambda,\xi}$ is extension closed in
$\mathfrak{g}\text{-}\mathrm{mod}$, claim \eqref{cor16.1}
follows directly from Theorem~\ref{thm6} and the corresponding
claim for $\mathbb{C}[[x_1,x_2,\dots,x_n]]$. Claim \eqref{cor16.2} follows
from claim \eqref{cor16.1} by standard induction on the 
lengths of $M$ and $N$.
\end{proof}

\subsection{A functor $\mathrm{F}$ 
from $\mathbb{C}[[x_1,x_2,\dots,x_n]]$-modules
to $\hat{\mathcal{C}}_{\lambda,\xi}$}\label{s2.2}

Let $V$ be a finite dimensional
$\mathbb{C}[[x_1,x_2,\dots,x_n]]$-module, hence $x_i$ acts via
some endomorphism which we call $X_i$.
For every $\mathbf{b}\in\mathbb{Z}^n$, $b_1+b_2+\dots+b_n=0$,
fix a copy, $V_{\mathbf{b}}$, of $V$ and consider the vector
space $\mathrm{F}V:=\oplus_{\mathbf{b}}V_{\mathbf{b}}$.

\begin{lemma}\label{lem7}
The space $\mathrm{F}V$ can be turned into a $\mathfrak{g}$-module,
where $e_{i,j}$ acts via $v\mapsto (X_j+(a_j+b_j)\mathrm{Id}_V)v\in V_{\mathbf{b}+\varepsilon_{i}-\varepsilon_{j}},$ and $h_i$ acts via
\begin{displaymath}
v\mapsto (X_i-X_{i-1}+(a_i+b_i-a_{i+1}-b_{i+1})\mathrm{Id}_V)v 
\in V_{\mathbf{b}}
\end{displaymath}
for $v\in V_{\mathbf{b}}$. Obviously, every $v\in V_{\mathbf{b}}$
annihilated by all $X_i$'s is a weight vector.
\end{lemma}

\begin{proof}
Consider the $\mathfrak{g}$-module $N(\mathbf{a})$ for  $\mathbf{a}$
as above. Then, for every $\mathbf{b}$ the defining relations of
$\mathfrak{g}$ (in generators $e_{i,i\pm1}$), 
applied to $\mathbf{v}_{\mathbf{b}}$, can be written as 
some polynomial equations in the $a_i$'s. Since \eqref{eq1} defines a 
$\mathfrak{g}$-module for any $\mathbf{a}$, these equations hold for any 
$\mathbf{a}$, that is they are actually formal identities. Write now 
$X_j+(a_j+b_j)\mathrm{Id}_V=A_j+B_j$, a sum of matrices, where
$A_j=X_j+a_j\mathrm{Id}_V$ and $B_j=b_j\mathrm{Id}_V$. Note that $A_j$ 
and $B_j$ commute. Now the defining relations for $\mathfrak{g}$ on 
$\mathrm{F}V$  reduce to our formal identities and hence are satisfied.
The formula for the action of $h_i$ is obtained by a direct computation. 
\end{proof}

Let $V$ and $V'$ be two finite dimensional
$\mathbb{C}[[x_1,x_2,\dots,x_n]]$-modules and let
$f:V\to V'$ be a homomorphism. Then $f$ extends diagonally to a
$\mathbb{C}$-linear map $\mathrm{F}f:\mathrm{F}V\to\mathrm{F}V'$. As $f$
commutes with all $X_i$, the map $\mathrm{F}f$ commutes with all
$e_{i,j}$ and hence defines a homomorphism of $\mathfrak{g}$-modules.
As a consequence, $\mathrm{F}$ becomes a functor from the category of
finite dimensional $\mathbb{C}[[x_1,x_2,\dots,x_n]]$-modules to the
category of $\mathfrak{g}$-modules. By construction, $\mathrm{F}$ is exact
and faithful. Furthermore, it sends the simple one-dimensional
$\mathbb{C}[[x_1,x_2,\dots,x_n]]$-module to $N(\mathbf{a})$,
which is an object of $\hat{\mathcal{C}}_{\lambda,\xi}$.
As $\hat{\mathcal{C}}_{\lambda,\xi}$ is extension closed in
$\mathfrak{g}\text{-}\mathrm{mod}$, we get that
the image of $\mathrm{F}$ belongs to $\hat{\mathcal{C}}_{\lambda,\xi}$. Hence we have a functor from the category of finite dimensional $\mathbb{C}[[x_1,x_2,\dots,x_n]]$-modules to $\hat{\mathcal{C}}_{\lambda,\xi}$.

Let $\mu_{\mathbf{b}}$
denote the weight of $\mathbf{v}_{\mathbf{b}}$.
Note that $\mathbf{b}\neq \mathbf{b}'$
implies $\mu_{\mathbf{b}}\neq \mu_{\mathbf{b}'}$.
It follows that $V_{\mathbf{b}}$ is the generalized
weight space of $\mathrm{F}V$ of weight $\mu_{\mathbf{b}}$.
If $\varphi:\mathrm{F}V\to \mathrm{F}V'$ is a
$\mathfrak{g}$-homomorphism, it must preserve the weight spaces and
hence it induces a linear map $f:V\to V'$ (on the component with
$\mathbf{b}=\mathbf{0}$). As $\varphi$ commutes with all $h_i$, the
map $f$ commutes with all operators $X_i-X_{i+1}$. As $\varphi$ commutes with
the element $C=(h_1+1)^2+4e_{21}e_{12}$, the map $f$ commutes with
the operator $((a_1+a_2+1)\mathrm{Id}_V+X_1+X_2)^2$ and thus
with all polynomials in this operator. If $a_1+a_2\not\in\mathbb{Z}$
(which we may assume by Proposition~\ref{prop3}),
then the latter operator is invertible and 
we get that $f$ commutes with the polynomial square root of it and hence
with $X_1+X_2$. This implies that $f$ commutes with all $X_i$.
Therefore $\varphi=\mathrm{F}f$ and thus the functor $\mathrm{F}$ is full.

To prove that $\mathrm{F}$ is an equivalence, we are left to show
that it is dense (i.e. essentially surjective). We first 
consider the case $n=3$.

\subsection{Commutativity of $D_{\lambda,\xi}$ and $\hat{D}_{\lambda,\xi}$ for $\mathfrak{sl}_3$}\label{s2.3}

In this subsection we assume that $\mathfrak{g}=\mathfrak{sl}_3$ and 
prove the following statements:

\begin{proposition}\label{prop9}
The algebras 
$D_{\lambda,\xi}$ and $\hat{D}_{\lambda,\xi}$ are commutative.
\end{proposition}

\begin{corollary}\label{cor10}
Let $U_0$ denote the centralizer of $\mathfrak{h}$ in $U$, $x,y\in U_0$,
$M\in \hat{\mathcal{C}}_{\lambda,\xi}$ and $m\in M$. Then we have
$x(y(m))=y(x(m))$.
\end{corollary}

\begin{corollary}\label{cor11}
Let $L$ be a simple object in $\hat{\mathcal{C}}_{\lambda,\xi}$
then $\dim \mathrm{Ext}^1_{\mathfrak{g}}(L,L)<\infty$.
\end{corollary}

Consider the subalgebra  $\Gamma$ of $U$ as defined in Subsection~\ref{s1.6}. If $M$ is a weight $\mathfrak{g}$-module, then every 
$M^{\mu}$ becomes a finite dimensional $\Gamma$-module. 
For the category $\Gamma\text{-}\mathrm{mod}$ of finite dimensional
$\Gamma$-modules we have the usual decomposition
\begin{displaymath}
\Gamma\text{-}\mathrm{mod}\cong
\bigoplus_{\chi:\Gamma\to \mathbb{C}}
\Gamma_{\chi}\text{-}\mathrm{mod},
\end{displaymath}
where $\Gamma_{\chi}\text{-}\mathrm{mod}$ denotes the full subcategory of
$\Gamma\text{-}\mathrm{mod}$, which consists of all $M$ annihilated by 
some power of $\mathrm{Ker}(\chi)$. 
The  category $\Gamma_{\chi}\text{-}\mathrm{mod}$
is equivalent to the category  of finite dimensional modules over
the completion $\Gamma_{\chi}$ of $\Gamma$ with respect to the
kernel $\mathrm{ker}(\chi)$ of $\chi$.

As $U$ is free as a right $\Gamma$-module, the induction functor
\begin{displaymath}
\mathrm{Ind}_{\Gamma}^{U}:=
U\otimes_{\Gamma}{}_-:
\Gamma\text{-}\mathrm{mod}\to
U\text{-}\mathrm{mod},
\end{displaymath}
from the category $\Gamma\text{-}\mathrm{mod}$ to the category 
$U\text{-}\mathrm{mod}$ of finitely generated $U$-modules
is exact and sends nonzero modules to nonzero modules. From \cite{Ov}
it even follows that any module in the image of
$\mathrm{Ind}_{\Gamma}^{U}$ has  finite length. 

Fix some $\mu\in\xi$. If $L$ is a simple object in
$\hat{\mathcal{C}}_{\lambda,\xi}$, then $L_{\mu}$ is one-dimensional
(as we assume that $L$ is pointed). Since $\Gamma L_{\mu}\subset L_{\mu}$,
it follows that $L_{\mu}$ is a simple $\Gamma$-module, which corresponds
to some character $\chi=\chi_{\mu}:\Gamma\to\mathbb{C}$. All other nonzero
weight spaces of $L$ are also simple $\Gamma$-modules, but not isomorphic
to $L_{\mu}$ (as $\Gamma$ contains $\mathfrak{h}$). As any object in
$\hat{\mathcal{C}}_{\lambda,\xi}$ has a filtration with subquotients isomorphic
to $L$, it follows that the assignment
\begin{displaymath}
\hat{\mathcal{C}}_{\lambda,\xi}\ni M\mapsto M_{\mu}\in
\Gamma_{\chi}\text{-}\mathrm{mod}
\end{displaymath}
defines a functor $\mathrm{G}_{\chi}$ from $\hat{\mathcal{C}}_{\lambda,\xi}$
to $\Gamma_{\chi}\text{-}\mathrm{mod}$. The functor $\mathrm{G}_{\chi}$
is by definition exact and sends simple modules to simple modules. From the
classical adjunction between induction and restriction, the left
adjoint of $\mathrm{G}_{\chi}$ is the functor $\mathrm{F}_{\chi}$, which
is the composition of $\mathrm{Ind}_{\Gamma}^{U}$ followed
by taking the maximal quotient, which belongs to $\hat{\mathcal{C}}_{\lambda,\xi}$
(the latter is well-defined as the induced module has finite length
and $\hat{\mathcal{C}}_{\lambda,\xi}$ is extension closed in
$U\text{-}\mathrm{mod}$).

\begin{lemma}\label{lem8-1}
The functor $\mathrm{F}_{\chi}$ sends simple modules to simple modules.
\end{lemma}

\begin{proof}
Let $L\in \hat{\mathcal{C}}_{\lambda,\xi}$ be simple. 
Then $L_{\mu}$ is a simple $\Gamma$-module and we have
\begin{equation}\label{eq1005}
\mathrm{Hom}_{\mathfrak{g}}(\mathrm{Ind}_{\Gamma}^{U}L_{\mu},L)\cong
\mathrm{Hom}_{\Gamma}(L_{\mu},L)\cong\mathbb{C}.
\end{equation}
Moreover, we claim that the multiplicity of $L$ in
$\mathrm{Ind}_{\Gamma}^{U}L_{\mu}$ is one, which then implies $\mathrm{F}_{\chi}L_{\mu}\cong L$ and completes the proof.
Because of \eqref{eq1005} we only have to show that this 
multiplicity is at most one. The verification of this claim is a 
technical computation using the Gelfand-Zetlin combinatorics for
$\mathfrak{sl}_3$. Our argument does not generalize to
$\mathfrak{sl}_n$, $n>3$.

As explained in Subsection~\ref{s1.6}, the character 
$\chi=\chi_{\mu}$ is represented by
a tableau of the following form:
\begin{equation}\label{eq2}
\begin{picture}(80.00,80.00)
\drawline(10.00,70.00)(70.00,70.00)
\drawline(10.00,50.00)(70.00,50.00)
\drawline(20.00,30.00)(60.00,30.00)
\drawline(30.00,10.00)(50.00,10.00)
\drawline(10.00,70.00)(10.00,50.00)
\drawline(30.00,70.00)(30.00,50.00)
\drawline(50.00,70.00)(50.00,50.00)
\drawline(70.00,70.00)(70.00,50.00)
\drawline(20.00,50.00)(20.00,30.00)
\drawline(40.00,50.00)(40.00,30.00)
\drawline(60.00,50.00)(60.00,30.00)
\drawline(30.00,30.00)(30.00,10.00)
\drawline(50.00,30.00)(50.00,10.00)
\put(20.00,60.00){\makebox(0,0)[cc]{{\tiny $a$}}}
\put(40.00,60.00){\makebox(0,0)[cc]{{\tiny $b$}}}
\put(60.00,60.00){\makebox(0,0)[cc]{{\tiny $b-1$}}}
\put(30.00,40.00){\makebox(0,0)[cc]{{\tiny $x$}}}
\put(50.00,40.00){\makebox(0,0)[cc]{{\tiny $b$}}}
\put(40.00,20.00){\makebox(0,0)[cc]{{\tiny $z$}}}
\end{picture}
\end{equation}
where $a,b,x,y\in\mathbb{C}$ and $z-x,z-b,x-a,x-b\not\in\mathbb{Z}$. 
For any $\nu\in\mathrm{supp}(L)$ the corresponding
$\chi_{\nu}$ is represented by a tableau $\tau'$ with the same
$a,b$ and some $x',z'$ such that 
$x-x'\in\mathbb{Z}$ and $z-z'\in\mathbb{Z}$.

Simple $\Gamma$-subquotients appearing in the module
$\mathrm{Ind}_{\Gamma}^{U}L_{\mu}$ (considered as $\Gamma$-module)
are indexed (with the corresponding multiplicities) by tableaux 
obtained from $\tau$ by adding arbitrary integers in the two bottom 
rows, that is to $x$, $b$ and $z$
(see \cite[Section~4]{Ov}). As $x-b\not\in\mathbb{Z}$, it follows that
$\tau$ itself appears in this set only once (namely when we add zeros).
This implies that the multiplicity of $L$ in
$\mathrm{Ind}_{\Gamma}^{U}L_{\mu}$ is at most one (note that the latter
is not true for $n>3$).
\end{proof}

\begin{lemma}\label{lem8}
The adjunction morphisms
\begin{displaymath}
\mathrm{adj}:\mathrm{F}_{\chi}\mathrm{G}_{\chi}\to
\mathrm{ID}_{\hat{\mathcal{C}}_{\lambda,\xi}}\quad
\text{ and }\quad
\mathrm{adj}':\mathrm{ID}_{\Gamma_{\chi}\text{-}\mathrm{mod}}
\to\mathrm{G}_{\chi}\mathrm{F}_{\chi}
\end{displaymath}
are surjective.
\end{lemma}

\begin{proof}
If $L$ is a simple object of $\hat{\mathcal{C}}_{\lambda,\xi}$,
then $\mathrm{G}_{\chi}L$ is simple, in particular,
nonzero. Therefore, $\mathrm{adj}_L$ is nonzero (as the image
of the nonzero identity morphism on $\mathrm{G}_{\chi}L$ under the
adjunction isomorphism), hence surjective. The claim about $\mathrm{adj}$
follows now by induction
on the length of a module and the Five-Lemma. The second statement follows completely analogously using Lemma~\ref{lem8-1}.
\end{proof}

For every $k\in\mathbb{N}$ denote by $\Gamma_{\chi}\text{-}\mathrm{mod}^k$
and $\hat{\mathcal{C}}_{\lambda,\xi}^k$ the full subcategories of
$\Gamma_{\chi}\text{-}\mathrm{mod}$
and $\hat{\mathcal{C}}_{\lambda,\xi}$, respectively, which consist of all
modules of Loewy length at most $k$. Then
$\Gamma_{\chi}\text{-}\mathrm{mod}^k$ is equivalent to the category of
finite dimensional modules over the finite dimensional algebra
$\Gamma_{\chi}/(\mathrm{ker}(\chi))^k$. The category
$\hat{\mathcal{C}}_{\lambda,\xi}^k$ is an abelian length category.

\begin{corollary}\label{cor8-2}
\begin{enumerate}[(i)]
\item \label{cor8-2.1} The pair $(\mathrm{F}_{\chi},\mathrm{G}_{\chi})$
restricts to an adjoint pair of functors between
$\Gamma_{\chi}\text{-}\mathrm{mod}^k$ and $\hat{\mathcal{C}}_{\lambda,\xi}^k$.
\item \label{cor8-2.2}
The functor $\mathrm{F}_{\chi}$ sends projective modules
from $\Gamma_{\chi}\text{-}\mathrm{mod}^k$ to projective modules
in $\hat{\mathcal{C}}_{\lambda,\xi}^k$.
\end{enumerate}
\end{corollary}

\begin{proof}
The exactness of $\mathrm{G}_{\chi}$ implies that $\hat{\mathcal{C}}_{\lambda,\xi}^k$ gets mapped
to $\Gamma_{\chi}\text{-}\mathrm{mod}^k$. That $\mathrm{F}_{\chi}$ maps $\Gamma_{\chi}\text{-}\mathrm{mod}^k$ to $\hat{\mathcal{C}}_{\lambda,\xi}^k$ 
follows from Lemma~\ref{lem8-1}. This implies claim \eqref{cor8-2.1}.
Then the functor $\mathrm{F}_{\chi}$ is left adjoint
to an exact functor and hence sends projective modules to projective modules
or zero. Since $\mathrm{F}_{\chi}$ does not annihilate simple modules, claim \eqref{cor8-2.2} follows.
\end{proof}

\begin{corollary}\label{cor8-3}
The functor $\mathrm{F}_{\chi}$ is full on projective modules. 
In particular, $\hat{D}_{\lambda,\xi}$ is  a quotient of
$\Gamma_{\chi}$.
\end{corollary}

\begin{proof}
Let $P\in \Gamma_{\chi}\text{-}\mathrm{mod}$ be projective and $\varphi\in \mathrm{Hom}_{\mathfrak{g}}(\mathrm{F}_{\chi}P,\mathrm{F}_{\chi}P)$. Thanks 
to Lemma \ref{lem8} and the projectivity of $P$, the 
morphism $\mathrm{G}_{\chi}\varphi$ fits into the commutative diagram
\begin{equation}\label{eq3}
\xymatrix{
P\ar@{.>}[d]^{\psi}\ar@{->>}[rr]^{\mathrm{adj}'_P} &&
\mathrm{G}_{\chi}\mathrm{F}_{\chi}P\ar[d]^{\mathrm{G}_{\chi}\varphi} \\
P\ar@{->>}[rr]^{\mathrm{adj}'_P} && \mathrm{G}_{\chi}\mathrm{F}_{\chi}P
}
\end{equation}
for some $\psi:P\to P$. Applying $\mathrm{F}_{\chi}$ gives the square on 
the left hand side of the commutative diagram:
\begin{equation}\label{eq4}
\xymatrix{
\mathrm{F}_{\chi}P\ar@{.>}[d]^{\mathrm{F}_{\chi}\psi}
\ar@{->>}[rr]^{\mathrm{F}_{\chi}\mathrm{adj}'_P}
&& \mathrm{F}_{\chi}\mathrm{G}_{\chi}\mathrm{F}_{\chi}P
\ar[d]^{\mathrm{F}_{\chi}\mathrm{G}_{\chi}\varphi}
\ar@{->>}[rr]^{\mathrm{adj}_{\mathrm{F}_{\chi}P}}
&& \mathrm{F}_{\chi}P \ar[d]^{\varphi}\\
\mathrm{F}_{\chi}P\ar@{->>}[rr]^{\mathrm{F}_{\chi}\mathrm{adj}'_P} && \mathrm{F}_{\chi}\mathrm{G}_{\chi}\mathrm{F}_{\chi}P
\ar@{->>}[rr]^{\mathrm{adj}_{\mathrm{F}_{\chi}P}}
&&\mathrm{F}_{\chi}P.
}
\end{equation}
By the adjointness property, the compositions in each row are the identity
maps. Hence $\varphi=\mathrm{F}_{\chi}\psi$ and the claim follows.
\end{proof}

\begin{proof}[Proof of Proposition \ref{prop9}:]
For the algebra $\hat{D}_{\lambda,\xi}$ the claim follows 
from Corollary~\ref{cor8-3} and the fact that the algebra  
$\Gamma_{\chi}$ is commutative. Thus the algebra 
${D}_{\lambda,\xi}$ is commutative as it is a quotient of the
commutative algebra $\hat{D}_{\lambda,\xi}$.
\end{proof}

\begin{proof}[Proof of Corollary \ref{cor10}:]
It is enough to prove the statement under the assumption that
$m\in M_{\mu}$, $\mu\in \xi$. Define
$I:= \bigcap_{N\in \hat{\mathcal{C}}_{\lambda,\xi}}\mathrm{Ann}_{U_0}N_{\mu},$
which is an ideal of $U_0$. Let $L$ be a simple object of
$\hat{\mathcal{C}}_{\lambda,\xi}$ and $\pi:U_0\to \mathbb{C}$ be the
homomorphism defining the simple $U_0$-module $L_{\mu}$. Then
$\pi(I)=0$ and hence we have the induced morphism
$\overline{\pi}:U_0/I\to \mathbb{C}$. By the
PBW Theorem the algebra $U$ is free both as a left and as a right
$U_0$-module. Hence the usual induction and restriction define
an equivalence between $\hat{\mathcal{C}}_{\lambda,\xi}$ and the category
of finite-dimensional modules over the completion $\mathtt{U}$ of
$U_0/I$ with respect to the kernel of $\overline{\pi}$.
By Proposition~\ref{prop9},
the algebra $\mathtt{U}$ is commutative. The claim follows.
\end{proof}

\begin{proof}[Proof of Corollary \ref{cor11}:]
As mentioned in Subseection\ref{s1.6}, 
the algebra $\Gamma$ is a polynomial algebra in
five  variables (see \cite{DFO}).
Hence the quotient $\Gamma_{\chi}/(\mathrm{ker}(\chi))^2$
is finite-dimensional. As $\hat{D}_{\lambda,\xi}$ is  a quotient of
$\Gamma_{\chi}$ (Corollary~\ref{cor8-3}), it follows
that $\dim \mathrm{Ext}^1_{\mathfrak{g}}(L,L)$ does not exceed
the dimension of $\Gamma_{\chi}/(\mathrm{ker}(\chi))^2$.
\end{proof}

\subsection{Density of the functor $\mathrm{F}$ (proof of 
Theorem~\ref{thm6})}\label{s2.4}

In this subsection we complete the proof of Theorem~\ref{thm6}.
For this we are left to show that the functor $\mathrm{F}$ is dense 
(i.e. essentially surjective). We will
prove this by induction on $n$. By Proposition \ref{prop3} we might 
restrict to blocks $\hat{\mathcal{C}}_{\la,\xi}$, where 
$\xi\in \mathfrak{h}^*/\mathbb{Z}$ is represented by an element 
$\mu$ such that $\mu(h_i)=a_i-a_{i+1}$ with
\begin{equation}
\label{assumption}
a_i+a_{i+1}\not\in\mathbb{Z}\quad\text{ for all }i.
\end{equation}
In the following we will only consider such blocks, hence assume the functor $\mathrm{F}$ has values in a block for which the condition \eqref{assumption}
is satisfied. We begin with the starting point of our induction:

\begin{lemma}\label{lem12}
Let $n=2$ and $\lambda$ be as in case \eqref{case1}.
Then the functor $\mathrm{F}$ is dense.
\end{lemma}

\begin{remark}
{\rm 
As stated in the paragraph after Theorem~\ref{thm6}, this theorem is known 
to be true in case $n=2$, hence we do not need to prove Theorem \ref{thm6}
in this case. However, Lemma~\ref{lem12} is a stronger statement than just
Theorem \ref{thm6}, it gives us some information about the functor
$\mathrm{F}$, which we will need later on to prove the general case.
Our assumption \eqref{assumption} excludes case \eqref{case2} and we 
do not know if the functor $\mathrm{F}$ is dense in that case. For our 
induction argument it is enough to have the density as formulated in 
Lemma \ref{lem12}.
}
\end{remark}

\begin{proof}
Let $M\in \hat{\mathcal{C}}_{\lambda,\xi}$. Then $\mu=a_1-a_2\in\xi$. 
Set  $V=M_{\mu}$. Let $Y_1$ and $Y$ be the linear operators on $V$ 
representing the actions of the element $h_1$ and the 
$\mathfrak{sl}_2$-Casimir element $C$ 
(see Subsection~\ref{s2.2}), respectively. As we are in case~\eqref{case1}, 
the element $\lambda$ is regular and we have that the (unique) eigenvalue
of $Y$ is nonzero, in particular, $Y$ is invertible. Let $Y'$
denote any square root of $Y$, which is a polynomial in $Y$.
Then $Y'$ commutes with $Y_1$ and $Y$. Set
\begin{eqnarray}\label{X1X2}
X_1:=\frac{Y_1+Y'-\mathrm{Id}_V}{2}-a_1\mathrm{Id}_V,&&
X_2:=\frac{Y'-Y_1-\mathrm{Id}_V}{2}-a_2\mathrm{Id}_V.
\end{eqnarray}
Then both $X_1$ and $X_2$ are polynomials in $Y_1$ and $Y$, in particular,
$X_1$ and $X_2$ commute. A direct calculation, using the definition of 
the functor $\mathrm{F}$, shows that the action of $h_1$ on
$(\mathrm{F}V)_{\mu}$ is given by $Y_1$ and the action of $C$
on $(\mathrm{F}V)_{\mu}$ is given by $Y$. Since any module
in $\hat{\mathcal{C}}_{\lambda,\xi}$ is uniquely determined 
(up to isomorphism) by the
actions of $h_1$ and $C$ (see for example \cite[Chapter~3]{Maz2}),
it follows that $\mathrm{F}V\cong M$. This completes the proof.
\end{proof}

Let $M\in \hat{\mathcal{C}}_{\lambda,\xi}$, let $Y$ be the linear operator
representing the action of $C$ on $V=M_{\mu}$, and $Y_1,Y_2,\dots,Y_n$ be the
linear operators representing the actions of $h_1,h_2,\dots,h_n$ on
$M_{\mu}$, respectively. Define $X_1$ and $X_2$ by \eqref{X1X2} and then
inductively define $X_{i+1}=X_i-Y_i-(a_i-a_{i+1})\mathrm{Id}_V$, 
$i=2,3,\dots,n-1$. Then $X_1,X_2,\dots,X_n$ are commuting nilpotent
linear operators on $V$.

The $U(\mathfrak{sl}_{n-1})$-module $M'=U(\mathfrak{sl}_{n-1})M_{\mu}$ 
is cuspidal. Now either $n=3$ and $a_1+a_2\not\in\mathbb{Z}$, then the
$U(\mathfrak{sl}_{2})$-module $U(\mathfrak{sl}_{2})M_{\mu}$
has non-integral central character, which allows us to use 
Lemma~\ref{lem12}, or $n>3$ where we can use the induction hypothesis 
to conclude that the module $M'$ is in the image of the functor 
$\mathrm{F}$ for the algebra $U(\mathfrak{sl}_{n-1})$.
By definition of the functor $\mathrm{F}$, we have 
$M'\cong N':=\oplus_{\mathbf{b}}V_{\mathbf{b}}$, where $\mathbf{b}$ runs
through all elements $(b_1,b_2,\ldots b_{n-1},0)\in\mathbb{Z}^n$ such that
$b_1+b_2+\dots+b_{n-1}=0$; and according to Lemma \ref{lem7} 
the element $e_{i,j}$, $i,j<n$, $i\neq j$, acts on $N'$ by 
mapping $v\in V_{\mathbf{b}}$ to  $(X_j+(a_j+b_j)\mathrm{Id}_V)v\in
V_{\mathbf{b}+\varepsilon_{i}-\varepsilon_{j}}$. To complete the proof it 
is enough to show the following unique extension property (since from 
the uniqueness it will then follow that this extension is isomorphic to 
either of $M$ and $\mathrm{F}V$, implying $M\cong \mathrm{F}V$):

\begin{proposition}\label{prop1001}
The pair $(N',X_n)$ {\bf uniquely} extends
to a cuspidal $\mathfrak{g}$-module. That is,
under the assumption \eqref{assumption} there is a unique up to 
isomorphism cuspidal $\mathfrak{g}$-module $N$ such that 
$N=UN_{\mu}$, $U(\mathfrak{sl}_{n-1})N_{\mu}=N'$ and which gives the
linear operator $X_n$ when computed as above.
\end{proposition}

\begin{proof}
Since $a_n\not\in\mathbb{Z}$, the endomorphism 
$X_n+(a_n+b_n)\mathrm{Id}_V$ 
is invertible for all $b_m\in\mathbb{Z}$.
As the action of $e_{n-1,n}$ on $N$ is bijective, we may fix
bases for all weight spaces of $N$ such that the $U(\mathfrak{sl}_{n-1})$-action on $N'$ is as described above and the 
action of $e_{n-1,n}$ is given by the definition of $\mathrm{F}V$.
Now we have to define actions of all
$e_{i,i+1}$, $i=1,2,\dots,n-2$, and all $e_{i+1,i}$, $i=1,2,\dots,n-1$ on $N$.
If $i<n-2$, then $e_{i,i+1}$ and $e_{n-1,n}$ commute and hence the action
of $e_{i,i+1}$ extends uniquely from $N'$ to the rest of $N$.
Similarly with the action of all $e_{i+1,i}$, $i=1,2,\dots,n-2$. Pictorially this can be illustrated as follows:

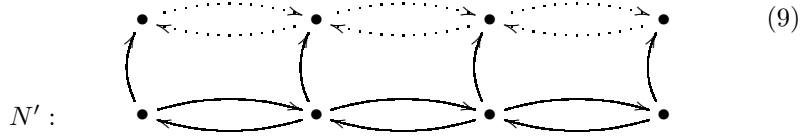
\begin{equation}\label{eq11}
\xymatrix{
&\bullet\ar@/^/@{.>}[rr] &&
\bullet\ar@/^/@{.>}[ll]\ar@/^/@{.>}[rr]&&
\bullet\ar@/^/@{.>}[ll]\ar@/^/@{.>}[rr]&&
\bullet\ar@/^/@{.>}[ll]\\
N':&\bullet\ar@/^/[u]\ar@/^/[rr]&&
\bullet\ar@/^/[u]\ar@/^/[rr]\ar@/^/[ll]&&
\bullet\ar@/^/[u]\ar@/^/[rr]\ar@/^/[ll]&&
\bullet\ar@/^/[u]\ar@/^/[ll]\\
}
\end{equation}
The $\bullet$'s represent weight spaces with fixed basis vectors, with $N'$ represented at the bottom. The arrows pointing up indicate the action of $e_{n-1,n}$, right and left arrows represent the actions of $e_{i,i+1}$ 
and $e_{i+1,i}$, respectively. Solid arrows represent already known 
actions, and dotted indicate the actions for which we 
show that they are uniquely defined.

This leaves us with the elements $e_{n-2,n-1}$ and $e_{n,n-1}$. 
That the action of these elements is uniquely defined, follows from 
Lemmata~\ref{lem1002} and \ref{lem1003} below.
\end{proof}

\begin{lemma}\label{lem1002}
Under the assumption \eqref{assumption} there is a unique way to
define the action of $e_{n-2,n-1}$ on $N$. 
\end{lemma}

\begin{proof}
Consider the following picture:
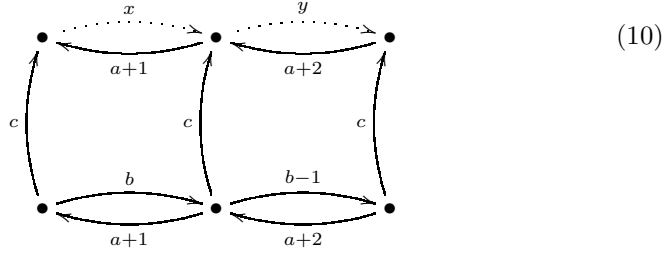
\begin{equation}\label{eq11n}
\xymatrix{
\bullet\ar@/^/@{.>}[rr]^{x} &&
\bullet\ar@/^/[ll]^{a+1}\ar@/^/@{.>}[rr]^{y} &&
\bullet\ar@/^/[ll]^{a+2}\\ \\
\bullet\ar@/^/[uu]^{c}\ar@/^/[rr]^{b} &&
\bullet\ar@/^/[uu]^{c}\ar@/^/[rr]^{b-1}\ar@/^/[ll]^{a+1}
&& \bullet\ar@/^/[uu]^{c}\ar@/^/[ll]^{a+2}\\
}
\end{equation}
Arrows pointing up indicate again the action of
$e_{n-1,n}$ (the labels are the corresponding linear maps), but 
right and left arrows represent now the actions of $e_{n-2,n-1}$ 
and $e_{n-1,n-2}$, respectively. Solid arrows represent already known 
actions, and dotted  indicate the actions for which we would
like to show that they are uniquely defined. 
There are $k,l,m\in \mathbb{Z}$ such that the coefficient $a$, $b$ and 
$c$ are of the form $X_{n-2}+(a_{n-2}+k)\mathrm{Id}_V$, 
$X_{n-1}+(a_{n-1}+l)\mathrm{Id}_V$ and $X_{n}+(a_{n}+m)\mathrm{Id}_V$,
respectively. By Corollary~\ref{cor10}, the action of the 
$\mathfrak{sl}_2$-Casimir element $(h_{n-1}+1)^2+4e_{nn-1}e_{n-1n}$ 
commutes with the action of all Cartan elements and with the action of
the $\mathfrak{sl}_2$-Casimir element $C$ from above. This implies that 
either of $x(a+1)$ and $y(a+2)$ commutes with any $X_i$. As
both, $a+1$ and $a+2$, are invertible, $x$ and $y$
commute with all $X_i$'s. From the relation
\begin{displaymath}
e_{n-2,n-1}e_{n-1,n-2}-e_{n-1,n-2}e_{n-2,n-1}=
e_{n-2,n-2}-e_{n-1,n-1}
\end{displaymath}
we obtain the equation
\begin{equation}\label{eq5}
x(a+1)-y(a+2)=a-b+1.
\end{equation}
From the Serre relation
\begin{displaymath}
e_{n-2,n-1}^2e_{n-1,n}-
2e_{n-2,n-1}e_{n-1,n}e_{n-2,n-1}+e_{n-1,n}e_{n-2,n-1}^2=0,
\end{displaymath}
taking into account that $c$ is invertible,
we obtain the equation
\begin{equation}\label{eq6}
xy-2yb+b(b-1)=0.
\end{equation}
Multiplying \eqref{eq6} by the invertible element $(a+1)$ and
inserting $x(a+1)$ from \eqref{eq5} we get:
\begin{displaymath}
y\big(y(a+2)+a-b+1\big)-2yb(a+1)+b(b-1)(a+1)=0.
\end{displaymath}
The latter factorizes as
\begin{equation}
\big(y(a+2)-(b-1)(a+1)\big)(y-b)=0.
\end{equation}

In the case of a simple module we know that $y=b$. Note that the equality
$b=(b-1)(a+1)/(a+2)$ implies $a+b\in \mathbb{Z}$. As we have assumed that
$a_i+a_j\not\in\mathbb{Z}$ for all $i,j$, it follows that the element
$(b-1)(a+1)/(a+2)$ acts nonzero on the corresponding weight space
of a simple module, hence invertible
on the corresponding weight space of any module. This implies
that $y=b$, that is $y$ is uniquely defined. This fact defines 
inductively the action of $e_{n-2,n-1}$ on the vector space
$(e_{n-1,n})^kN'$ for all $k\in\mathbb{N}$ uniquely. 

To determine the action of $e_{n-2,n-1}$ on the vector space
$e_{n-1,n}^{-k}N'$ for all $k\in\mathbb{N}$ we consider the following
picture with the same notation as in \eqref{eq11}:
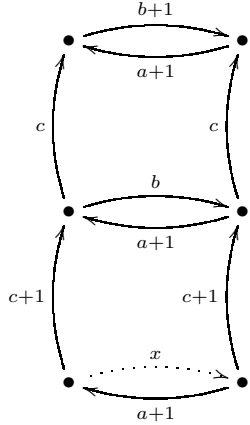
\begin{displaymath}
\xymatrix{
\bullet\ar@/^/[rr]^{b+1} &&
\bullet\ar@/^/[ll]^{a+1}\\ \\
\bullet\ar@/^/[uu]^{c}\ar@/^/[rr]^{b} &&
\bullet\ar@/^/[uu]^{c}\ar@/^/[ll]^{a+1}\\ \\
\bullet\ar@/^/[uu]^{c+1}\ar@/^/@{.>}[rr]^{x} &&
\bullet\ar@/^/[uu]^{c+1}\ar@/^/[ll]^{a+1}
}
\end{displaymath}
From the Serre relation
\begin{displaymath}
e_{n-1,n}^2e_{n-2,n-1}-
2e_{n-1,n}e_{n-2,n-1}e_{n-1,n}+e_{n-2,n-1}e_{n-1,n}^2=0,
\end{displaymath}
using the invertibility of $c$ and $c+1$ we get  $x-2b+(b+1)=0$. 
Hence there is a unique solution $x=b-1$. Inductively one defines 
the only possible action of $e_{n-2,n-1}$ on the vector space
$e_{n-1,n}^{-k}N'$ for all $k\in\mathbb{N}$. This completes the proof.
\end{proof}

\begin{lemma}\label{lem1003}
There is a unique way to define the action of $e_{n,n-1}$ on $N$.
\end{lemma}

\begin{proof}
To determine this action of $e_{n,n-1}$ on $N$ we consider the following
picture with the same notation as in \eqref{eq11}:
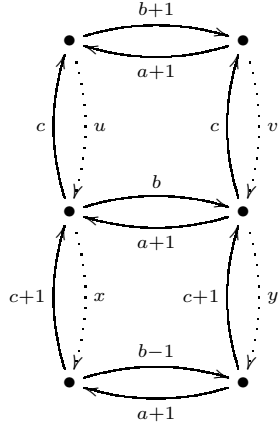
\begin{displaymath}
\xymatrix{
\bullet\ar@/^/@{.>}[dd]^{u}\ar@/^/[rr]^{b+1} &&
\bullet\ar@/^/[ll]^{a+1}\ar@/^/@{.>}[dd]^{v}\\ \\
\bullet\ar@/^/[uu]^{c}\ar@/^/[rr]^{b}\ar@/^/@{.>}[dd]^{x} &&
\bullet\ar@/^/[uu]^{c}\ar@/^/[ll]^{a+1}\ar@/^/@{.>}[dd]^{y}\\ \\
\bullet\ar@/^/[uu]^{c+1}\ar@/^/[rr]^{b-1} &&
\bullet\ar@/^/[uu]^{c+1}\ar@/^/[ll]^{a+1}
}
\end{displaymath}
Here all right arrows are now  determined (in particular, by 
Lemma~\ref{lem1002}) and we
have to figure out the down arrows, representing the action of
$e_{n,n-1}$. Similarly to the arguments above we obtain that
the elements $x$, $y$, $u$ and $v$ commute with all $X_i$'s.
As the elements $e_{n,n-1}$ and $e_{n-2,n-1}$, acting in the
lower square, must commute, we obtain $y=x(b-1)b^{-1}$. From the relation
\begin{displaymath}
e_{n-1,n}e_{n,n-1}-e_{n,n-1}e_{n-1,n}=
e_{n-1,n-1}-e_{n,n}
\end{displaymath}
we obtain $u=xc^{-1}(c+1)-bc^{-1}+1$, and $v=yc^{-1}(c+1)-bc^{-1}+1+c^{-1}$.
As the elements $e_{n,n-1}$ and $e_{n-2,n-1}$ acting in the upper square
must commute, we have $bu=v(c+1)$, which gives a linear equation on
$x$ with nonzero coefficients. This equation has a unique solution
(which is easily verified to be $x=b$).
\end{proof}

We proved in fact the following:

\begin{corollary}\label{cor15}
Assume that $n>2$ and that $a_i+a_j\not\in\mathbb{Z}$ for all $i,j$.
Then the functor $\mathrm{F}$ is an equivalence. In particular,  Theorem~\ref{thm6}\eqref{thm6.1} holds.
\end{corollary}

\begin{proof}[Proof of Theorem~\ref{thm6}\eqref{thm6.2}]
The category $\mathcal{C}_{\lambda,\xi}$ is defined inside
$\hat{\mathcal{C}}_{\lambda,\xi}$ by the condition that
the action of $\mathfrak{h}$ is diagonalizable. On modules in the image of
$\mathrm{F}$, the action of $h_i$, $1\leq i\leq n-1$, is given by
the linear operators
\begin{displaymath}
(a_i-a_{i+1})\mathrm{Id}_V+X_i-X_{i+1}.
\end{displaymath}
Hence such module is an object of $\mathcal{C}_{\lambda,\xi}$ if and only 
if the matrices $X_i-X_{i+1}$
are zero. This means that $D_{\lambda,\xi}$ is the quotient
of $\mathbb{C}[[x_1,x_2,\dots,x_n]]$ modulo the ideal generated by
the elements $x_i-x_{i+1}$. This proves the claim of Theorem~\ref{thm6}\eqref{thm6.2} for $n>2$. 
For $n=2$ it is known anyway.
\end{proof}

\section{The regular case as a deformation over the singular case}\label{s3}

The aim of this section is to make a first step in the proof of  Theorem~\ref{thmmain}\eqref{thmmain.3} and \eqref{thmmain.4}.
Here we will show that regular integral blocks of the category
of cuspidal modules can be considered as deformations of certain
finite dimensional algebras over singular blocks.

\subsection{The algebras $A^k$}

Thanks to Theorem \ref{thm6} we have a complete and explicit description of 
the categories of (generalized) weight modules in the Cases (I) and (II). 
In case (III), the classification theorem (Corollary \ref{cor2}) suggests 
a connection with the representation theory of highest weight modules for 
$\mathfrak{sl}_n$. Namely the modules $L(w\cdot\la)$ with 
$w\in W^{\operatorname{short}}$ are precisely the non-trivial simple 
modules in the regular integral block $\mathcal{O}_{\lambda}^\mathfrak{p}$ 
of the $\mathfrak{p}$-parabolic category  $\cO$ for $\mathfrak{sl}_n$ 
with  respect to the standard parabolic subalgebra 
with Levi factor $\mathfrak{sl}_{n-1}$. The category 
$\cO_{\lambda}^\mathfrak{p}$ is equivalent to the category of 
finite dimensional modules over the algebra $\tilde{A}^k$ for $k=n-1$, 
where $\tilde{A}^k$ is defined as the quotient of the path
algebra of the following quiver with $k+1$ vertices:
\begin{displaymath}
\xymatrix{
{0}\ar@/^/@{->}[r]^{a_0}
&\ar@/^/@{->}[l]^{b_0}
{1}\ar@/^/[r]^{a_1}&\ar@/^/[l]^{b_1}{2}
\ar@/^/[r]^{a_2}&\ar@/^/[l]^{b_2}\ar@/^/[r]^{a_3}{3}
&\cdots\ar@/^/[l]^{b_3}\ar@/^/[r]^{a_{k-1}}
&\ar@/^/[l]^{b_{k-1}}{k}
}
\end{displaymath}
modulo the relations $a_{i+1}a_i=0=b_{i}b_{i+1}$ and $b_ia_i=a_{i-1}b_{i-1}$
(whenever the expression makes sense) and the additional relation
$b_0a_0=0$. The algebra $\tilde{A}^k$ is a Koszul quasi-hereditary
algebra with a simple preserving duality (see e.g. \cite{BS2} where 
this algebra is called $K_{\Lambda}$, where $\Lambda$ is the block 
containing the finite weights with one up and $k-1$ downs). This algebra 
also describes blocks of the Temperley-Lieb algebra (\cite{Mar}, 
\cite{We}, \cite{KX2}), blocks of the category $\mathcal{O}$ for the 
Virasoro algebra (\cite{BNW}) and appears in the theory of rational 
representations for the algebraic group $\mathrm{GL}_2$ (\cite{MT}, \cite{Xi}).

Let $e_i\in \tilde{A}^k$ denote the trivial path at the vertex 
$i\in\{0,1,\dots,k\}$ (these are the primitive idempotents). 
We are interested in the subalgebra  $A^{k}=e\tilde{A}^ke$, where $e=\sum_{i=1}^ke_i$. In other words, for $k>1$ we denote by $A^{k}$ 
the quotient of the path algebra of the following  quiver  with $k$ vertices:
\begin{displaymath}
\xymatrix{
{1}\ar@/^/[r]^{a_1}&\ar@/^/[l]^{b_1}{2}
\ar@/^/[r]^{a_2}&\ar@/^/[l]^{b_2}\ar@/^/[r]^{a_3}{3}
&\cdots\ar@/^/[l]^{b_3}\ar@/^/[r]^{a_{k-1}}
&\ar@/^/[l]^{b_{k-1}}{k}
}
\end{displaymath}
modulo the relations $a_{i+1}a_i=0=b_{i}b_{i+1}$ and $b_ia_i=a_{i-1}b_{i-1}$
(whenever the expression makes sense) in  case $k>2$ and
$a_1b_1a_1=0=b_1a_1b_1$ in case $k=2$. Set $A^1=\mathbb{C}[x]/(x^2)$.

Note that the algebra $A^k$ is self-injective and even symmetric.
As a centralizer subalgebra of a quasi-hereditary algebra with duality,
it is cellular (\cite[Proposition~4.3]{KX1}, see also \cite{BS1} for an
explicit graded cellular basis). The algebra  $A^{k}$ belongs 
to the family of symmetric cellular algebras called generalized Khovanov 
algebras in \cite{BS1}. A realization as a convolution algebra using 
Springer fibres can be found in \cite{SW}. Replacing the double arrows
in the above quiver by simple edges, one obtains a tree and hence $A^{k}$
can be realized as the corresponding Brauer tree algebra (with $k-1$ edges
and no exceptional vertex in the classification of \cite[4.2]{Rickard}).

Let now $k\geq 2$ be fixed. We will need a few basic properties of $A^k$ 
which we recall now: For $1\leq i\leq k$ let $P_{i}=A^ke_i$ denote 
the indecomposable projective module corresponding to the $i$-th vertex. 
Then for $1\leq i,j\leq k$ the following holds:
\begin{displaymath}
\dim\mathrm{Hom}_{A^{k}}(P_{i},P_{j})=
\begin{cases}
2, & \text{if $i=j$};\\
1, & \text{if $i=j\pm 1$};\\
0, & \text{ otherwise}.
\end{cases}
\end{displaymath}
In particular, for $k>1$ the endomorphism algebra of the projective
generator $P=P_{1}\oplus P_{2}\oplus\cdots \oplus
P_{k}$ of $A^{k}\text{-}\mathrm{mod}$, which is naturally
identified with $A^{k}\cong (A^{k})^{\mathrm{op}}$, is generated
by unique (up to nonzero scalars) homomorphisms from
$P_{i}$ to $P_{i\pm 1}$
(whenever this makes sense). As a consequence, for $k>1$ the 
modules $P_{1}$ and $P_{k}$ have length three while all other 
indecomposable projectives have length four; all 
indecomposable projectives have Loewy length three.
Moreover the following is a basis of $A^{k}$:
\begin{displaymath}
\mathbf{B}_k=\{e_j,a_i,b_i,a_ib_i,b_1a_1\}
\end{displaymath}
(where $1\leq i\leq k-1$, $1\leq j\leq k$). In particular, $\dim A^{k}=4k-2$.

\subsection{Deformations}\label{s3.1}

Let $A$ be a finite dimensional algebra. Given a finite dimensional
vector space $U$ we denote by $\mathbb{C}_U=\mathbb{C}[[U]]$ the algebra
of $\mathbb{C}$-valued  formal functions on $U$. If $u_1,u_2,\dots,u_m$
is a basis of $U$, then $\mathbb{C}_U$ is naturally identified with the
algebra of formal power series $\mathbb{C}[[u_1,u_2,\dots,u_m]]$.
This algebra is local and we denote by $\mathfrak{m}$ its maximal ideal.
A {\it flat deformation} (or just a {\it deformation}) of $A$ over
$\mathbb{C}_U$ is a $\mathbb{C}_U$-algebra $A_U$ which is topologically
free as $\mathbb{C}_U$-module (i.e. isomorphic to some $V[[U]]$ for
some  vector space $V$) together with an isomorphism
$\varphi:A_U/\mathfrak{m}A_U\rightarrow A$ of algebras.
In particular $A_U\cong A[[U]]$ as a $\mathbb{C}_U$-module.
If $U$ has dimension $m$ then $A_U$ is an {\it $m$-parameter deformation}
of $A$. Two deformations $A_U$ and $A'_U$ are isomorphic if there is a $\mathbb{C}[[U]]$-isomorphism of algebras which is the identity
modulo $\mathfrak{m}$. A deformation $A_U$ is {\it trivial} if it is
isomorphic to $A[[U]]$ with the ordinary multiplication of formal
power series. Similarly one defines $k$-order deformations by replacing
$\mathbb{C}_U$ with $\mathbb{C}_U/\mathfrak{m}^{k+1}$. First order
deformations are also called {\it infinitesimal} deformations.

\subsection{Case \eqref{case3} as a deformation of $A^{n-1}$
over case \eqref{case2}}\label{s3.3}

The main result in this subsection is the following claim:

\begin{theorem}\label{thm21}
Let $\lambda$ be as in case \eqref{case3} and $\xi$ be such that
$\hat{\mathcal{C}}_{\lambda,\xi}$ is nonzero. Then
$\hat{D}_{\lambda,\xi}$ is a deformation of
$A^{n-1}$ over $\mathbb{C}[[x_1,x_2,\dots,x_n]]$; and
${D}_{\lambda,\xi}$ is a deformation of
$A^{n-1}$ over $\mathbb{C}[[x]]$.
\end{theorem}

In case $n=2$, the algebras $\hat{D}_{\lambda,\xi}$ and ${D}_{\lambda,\xi}$ 
are well-known (see e.g. \cite{Dr} or \cite[Chapter~3]{Maz2})
and the result is straightforward. Hence in what follows we assume $n>2$. 
Fix some $\mathbf{x}$ and $\mathttD$ as in
Theorem~\ref{thm1}. Recall the set $W^{\operatorname{short}}$ indexing the simple objects and set
\begin{displaymath}
\hat{L}_i:=\big( U_{S(\mathttD)}\otimes_U 
L((s_1s_{2}\cdots s_{i-1}s_{i})\cdot 
\lambda)\big)^{\Phi^{\mathttD}_{\mathbf{x}}},\quad
1\leq i\leq n-1.
\end{displaymath}
Fix some integral dominant singular $\lambda_i\in\mathfrak{h}^*$ with stabilizer $\langle s_i\rangle$ for the dot-action of $W$.
Let $\chi_\la$, $\chi_{\la_i}$ be the (under the Harish-Chandra isomorphism) corresponding central characters and denote by
$\mathcal{M}_{\lambda}$ and $\mathcal{M}_{\lambda_i}$ the full subcategories
of $\mathfrak{g}\text{-}\mathrm{mod}$ consisting of modules on which
$\mathrm{Ker}(\chi_\la)$ and $\mathrm{Ker}(\chi_{\la_i})$ act locally 
nilpotently, respectively. Let
\begin{displaymath}
\theta_i^{\mathrm{on}}:\quad\mathcal{M}_{\lambda}\to
\mathcal{M}_{\lambda_i} \quad\text{ and }\quad
\theta_i^{\mathrm{out}}:\quad\mathcal{M}_{\lambda_i}\to
\mathcal{M}_{\lambda}
\end{displaymath}
denote the corresponding projective functors
{\em translation on the wall} and {\em translation out of the wall},
respectively (see \cite{BG}, \cite{St} for details). These functors
are both left and right adjoint to each other and preserve the category of cuspidal modules as well as the category of weight modules. Recall (\cite[4.12]{Ja}, \cite[Lemma 5.3, Lemma A1]{St}) that 
$\theta_i^{\mathrm{on}}L \big((s_1s_{2}\cdots s_{j-1}s_{j})\cdot 
\lambda\big)$ is zero if $i\neq j$ and is a simple module if $i=j$.
From this it follows by standard arguments 
that $\theta_i^{\mathrm{on}}\hat{L}_j$
is zero if $i\neq j$ and is a simple module if $i=j$. 
We denote the latter module by $L_i$ (this is the simple object of $\hat{\mathcal{C}}_{\lambda_i,\xi_i}$).

Set $\xi_i=\mathrm{supp}(L_i)$. Then our functors restrict to 
biadjoint functors as follows:
\begin{displaymath}
\xymatrix{
\hat{\mathcal{C}}_{\lambda,\xi}
\ar@/^/[rr]^{\theta_i^{\mathrm{on}}}
&&\hat{\mathcal{C}}_{\lambda_i,\xi_i}
\ar@/^/[ll]^{\theta_i^{\mathrm{out}}}
},\quad
\xymatrix{
\mathcal{C}_{\lambda,\xi}
\ar@/^/[rr]^{\theta_i^{\mathrm{on}}}
&&\mathcal{C}_{\lambda_i,\xi_i}
\ar@/^/[ll]^{\theta_i^{\mathrm{out}}}
}.
\end{displaymath}
Furthermore, for any $i,i+1$ the composition
$\theta_{i+1}^{\mathrm{on}}\theta_{i}^{\mathrm{out}}$ is an equivalence with
inverse $\theta_{i}^{\mathrm{on}}\theta_{i+1}^{\mathrm{out}}$
(this follows from \cite[5.9]{Ja}, see \cite[(6.1),(6.2)]{St}).
By \cite{BG} for every $i$ we also have isomorphisms of functors
\begin{equation}\label{eq19}
\theta_i^{\mathrm{on}}\theta_i^{\mathrm{out}}\cong
\mathrm{Id}_{\hat{\mathcal{C}}_{\lambda_i,\xi_i}}\oplus
\mathrm{Id}_{\hat{\mathcal{C}}_{\lambda_i,\xi_i}}.
\end{equation}
From now on we fix such an isomorphism and let $p_1$ and $p_2$ denote 
the projection onto the first and second summands, respectively.

For $m\in\mathbb{N}$ consider the categories
$\hat{\mathcal{C}}_{\lambda_i,\xi_i}^m$ and
$\mathcal{C}_{\lambda_i,\xi_i}^m$
as in Subsection~\ref{s1.5}.
By Theorem~\ref{thm6}, the categories
$\hat{\mathcal{C}}_{\lambda_i,\xi_i}$ and
$\mathcal{C}_{\lambda_i,\xi_i}$ are equivalent to the categories
of finite dimensional modules over some local ring. In either case let
$\mathfrak{m}$ denote the maximal ideal.

\begin{lemma}\label{lem22}
The categories $\hat{\mathcal{C}}_{\lambda_i,\xi_i}^m$ and
$\mathcal{C}_{\lambda_i,\xi_i}^m$ coincide with the full subcategories
of $\hat{\mathcal{C}}_{\lambda_i,\xi_i}$ and $\mathcal{C}_{\lambda_i,\xi_i}$,
respectively, consisting of modules annihilated by $\mathfrak{m}^m$.
\end{lemma}

\begin{proof}
This follows directly from the definitions.
\end{proof}

Denote by ${\hat{\mathcal{C}}}_{\lambda,\xi}^{(m)}$ 
and ${{\mathcal{C}}}_{\lambda,\xi}^{(m)}$
the full subcategories 
of $\hat{\mathcal{C}}_{\lambda,\xi}$ and
${\mathcal{C}}_{\lambda,\xi}$ which for every $i$ are mapped 
to  $\hat{\mathcal{C}}^m_{\lambda_i,\xi_i}$ and
${\mathcal{C}}^m_{\lambda_i,\xi_i}$ by the corresponding
translation to the wall, respectively. 
The category ${\hat{\mathcal{C}}}_{\lambda,\xi}^{(m)}$
is abelian and the translation 
functors restrict, thanks to \eqref{eq19}, to (biadjoint, exact) 
functors between ${\hat{\mathcal{C}}}_{\lambda,\xi}^{(m)}$ and $\hat{\mathcal{C}}^m_{\lambda_i,\xi_i}$.  These categories then have enough projectives, and $\theta_i^{\mathrm{out}}$ (as left adjoint to an exact functor) maps
projective objects to projective objects. In particular, if we denote
by $\hat{R}(m)$ the indecomposable projective module
in $\hat{\mathcal{C}}_{\lambda_1,\xi_1}^m$, and set
\begin{eqnarray*}
\hat{P}_i(m):=\mathrm{F}_i \hat R(m), &\text{where}&\mathrm{F}_i=\theta_i^{\mathrm{out}}\theta_{i}^{\mathrm{on}}
\theta_{i-1}^{\mathrm{out}}\cdots
\theta_{2}^{\mathrm{out}}\theta_{2}^{\mathrm{on}}
\theta_{1}^{\mathrm{out}},
\end{eqnarray*}
then $\hat{P}(m):= \bigoplus_{i=1}^{n-1}\hat{P}_i(m)$ is a (minimal) 
projective generator of ${\hat{\mathcal{C}}}_{\lambda,\xi}^{(m)}$. 
Let $\hat{E}^{(m)}$ be its endomorphism ring and $\hat{E}_i^{(m)}$ 
the endomorphism ring of  $\hat{P}_i(m)$. By Section \ref{s1.5} we have $\hat{R}:=\hat{D}_{\la_1,\xi_1}=\varprojlim 
\operatorname{End}_{\mathfrak{g}}(\hat{R}(m))$ 
and $\hat{D}_{\la,\xi}=\varprojlim \hat{E}^{(m)}$. Let 
$\hat{P}_i(\infty)=\varprojlim \hat{P}_i(m)$, 
$\hat{P}(\infty)=\varprojlim \hat{P}(m)$ and
$\hat{R}^{(m)}=\operatorname{End}_{\mathfrak{g}}(\hat{R}(m))$. Define algebra
homomorphisms
\begin{eqnarray}\label{freeness1}
\Psi_i:\quad \hat{R}\to
 \varprojlim \hat{E}_i^{(m)},&&f\mapsto F_i(f).
\end{eqnarray}
and $\Psi=\bigoplus_{i=1}^n\Psi_i:\hat{R}\to \hat{D}_{\la,\xi}$.

For $1\leq i\leq n-1$ let $\alpha_{i}$ be the adjunction morphism 
from the identity functor to the composition $\theta_{i+1}^{\rm{out}}\theta_{i+1}^{\rm{on}}$ and $\beta_{i}$ 
the adjunction morphism in the opposite direction.

\begin{proposition}\label{freeness}
\begin{enumerate}[(i)]
\item\label{freeness.1} For $1\leq i\leq n-1$ the map $\Psi_i$ is an inclusion.
It turns $\varprojlim E_i^{(m)}$ into an $\hat{R}$-module which is free of 
rank $2$ with basis $\operatorname{id}=\operatorname{id}_i$, and the 
compositions $\alpha_{i-1}\beta_{i-1}$ of adjunction morphisms for $i>1$ 
and  $\beta_{1}\alpha_{1}$ otherwise.
\item\label{freeness.2} The map $\Psi$ is an inclusion. It turns 
$\hat{D}_{\la,\xi}$ into a 
free left and right $\hat{R}$-module of rank $4n-2$. A basis of this module is 
given by the elements $\operatorname{id}_j$, $\alpha_i$, $\beta_i$, $\alpha_i\beta_i$, $\beta_1\alpha_1$, for $1\leq i\leq n-2$ and 
$1\leq j\leq n-1$.
\end{enumerate}
\end{proposition}

\begin{proof}
Since $\mathrm{F}_i$ is the composition of 
$\theta_{i}^{\mathrm{out}}$ and  equivalences, it is enough to prove 
the first statement for $i=1$. For all other $i$ the proof is similar.
To show injectivity we assume $\Psi_1(f)=\Psi_1(f')$, that is 
$\mathrm{F}_1(f)=\mathrm{F}_1(f')$. Then $\theta_{1}^{\mathrm{on}}\theta_{1}^{\mathrm{out}}(f)=
\theta_{1}^{\mathrm{on}}\theta_{1}^{\mathrm{out}}(f')$, 
hence $f=f'$ by \eqref{eq19}. This implies that $\Psi_1$ is injective. 

The space  $\varprojlim E_1^{(m)}$ becomes an $\hat{R}$-module by 
setting $f.g=\theta_{1}^{\mathrm{out}}(f)\circ g$ for 
$g\in \varprojlim E_i^{(m)}$, $f\in \hat{R}$. Moreover, $\Psi_1$ becomes an 
$\hat{R}$-module morphism. We claim that the map 
$g\mapsto p_1\theta_{1}^{\mathrm{on}}(g)$ defines a split $S$ 
of $\Psi_1$. Since
\begin{eqnarray*}
S(f.g)&=&p_1\theta_{1}^{\mathrm{on}}(\theta_{1}^{\mathrm{out}}(f)\circ g)\\
&=&p_1(\theta_{1}^{\mathrm{on}}
\theta_{1}^{\mathrm{out}}(f)\circ\theta_{1}^{\mathrm{on}}(g))\\
&=&p_1((f\oplus f)\circ\theta_{1}^{\mathrm{on}} g)\\&=&f\circ p_1\theta_{1}^{\mathrm{on}}(g),
\end{eqnarray*}
this map is an $\hat{R}$-module homomorphism and obviously 
$S\circ\Psi_1$ is the identity on $\hat{R}$. It is now easy to verify 
that the map $g\mapsto p_2\theta_{1}^{\mathrm{on}}(g)$ defines a 
complement and  so $\varprojlim E_1^{(m)}\cong \hat{R}\oplus \hat{R}$. 
By direct calculations one verifies that the composition $\beta_1\alpha_1$ 
can be chosen as a second basis vector. Claim
\eqref{freeness.1} follows. 

Using again adjunctions and the fact that $\theta_{i+1}^{\mathrm{on}}\theta_{i}^{\mathrm{out}}$ is an equivalence we obtain from Theorem \ref{thm6} the following:
\begin{eqnarray*}
\HOM(\hat{P}_i(\infty),\hat{P}_{i+1}(\infty))&=&
\HOM(\hat{P}_i(\infty),\theta_{i+1}^{\mathrm{out}}
\theta_{i+1}^{\mathrm{on}}\hat{P}_{i}(\infty))\\
&=&
\HOM(\theta_{i+1}^{\mathrm{on}}\hat{P}_i(\infty),
\theta_{i+1}^{\mathrm{on}}\hat{P}_{i}(\infty))\\&\cong& \hat{R}.
\end{eqnarray*}
Hence, under the identification 
$\theta_{i+1}^{\mathrm{out}}\theta_{i+1}^{\mathrm{on}}
\hat{P}_i(\infty)=\hat{P}_{i+1}(\infty)$ the map
\begin{displaymath}
\hat{R}\rightarrow\HOM_{\mathfrak{g}}(\hat{P}_i(\infty),\hat{P}_{i+1}(\infty)),
\quad f\mapsto \alpha_i\circ f
\end{displaymath}
defines an isomorphism of $\hat{R}$-modules, where 
$f.g=\theta_{i+1}^{\mathrm{out}}\theta_{i+1}^{\mathrm{on}}(\Psi_i(f))\circ g$ 
for $f\in \hat{R}$, 
$g\in\HOM_{\mathfrak{g}}(\hat{P}_i(\infty),\hat{P}_{i+1}(\infty))$, since the 
natural transformation $\alpha_i$ satisfies
\begin{displaymath}
\theta_{i+1}^{\mathrm{out}}\theta_{i+1}^{\mathrm{on}}(\Psi_i(f))
\circ \alpha_i\circ g=\alpha_i\circ \Psi_i(f)\circ g=\alpha_i\circ (f.g).
\end{displaymath}
Therefore, $\HOM_{\mathfrak{g}}(\hat{P}_i(\infty),\hat{P}_{i+1}(\infty))$ 
is a free right $\hat{R}$-module of rank one with basis $\alpha_i$. 
Similarly, $\HOM_{\mathfrak{g}}(\hat{P}_{i+1}(\infty),\hat{P}_{i}(\infty))$ 
is a free left $\hat{R}$-module of rank one with basis $\beta_i$.
Finally we claim that
\begin{equation}\label{nohom}
\HOM_{\mathfrak{g}}(\hat{P}_i(\infty),\hat{P}_{j}(\infty))=
0\quad \text{if $|i-j|>1$}.
\end{equation}
By the standard properties of translation out of the wall,
the module  $\theta_i^{\mathrm{out}}L_i$ has simple top and simple socle, 
both isomorphic to $\hat{L}_i$. By adjunction,
\begin{displaymath}
\mathrm{Hom}_{\mathfrak{g}}(\theta_i^{\mathrm{out}}L_i,
\theta_{i\pm 1}^{\mathrm{out}}L_{i\pm 1})=
\mathrm{Hom}_{\mathfrak{g}}(L_i,
\theta_i^{\mathrm{on}}\theta_{i\pm 1}^{\mathrm{out}}L_{i\pm 1})\neq 0
\end{displaymath}
and hence $\theta_i^{\mathrm{out}}L_i$ has at least one composition 
factor isomorphic to $\hat{L}_{i\pm 1}$, whenever $i\pm 1$ makes sense. 
Comparing the character of $\theta_i^{\mathrm{out}}L_i$ with its 
$\Phi_{\mathbf{x}}^{\mathttD}$-twisted character, we conclude that
$\theta_i^{\mathrm{out}}L_i$ has length three if $i=1,n-1$
and length four otherwise and the mentioned above simple subquotients are
all simple subquotients of $\theta_i^{\mathrm{out}}L_i$. Since 
$\theta_i^{\mathrm{out}}$ is exact, Theorem \ref{thm6} implies that there is no 
simple composition factor of the form $\hat{L_i}$ appearing in  
$\hat{P}_{j}(\infty)$, hence the claim \eqref{nohom} follows, since  
$\hat{P}_{i}(\infty)$ is a limit of projective covers of $\hat{L}_i$. The 
problem with left and right $\hat{R}$-module structures is dealt with in 
Lemma~\ref{Rcentral} below. Hence $\hat{D}_{\la,\xi}$ is a free 
left $\hat{R}$-module of rank $4n-2$ with a basis as stated in the proposition.
\end{proof}

\begin{lemma}\label{Rcentral}
The image $I$ of $\Psi$ is a central subalgebra.
\end{lemma}

\begin{proof}
By the proof of Proposition \ref{freeness} it is enough to show 
that any $f\in \hat{R}$ satisfies $\Psi(f)\circ\alpha_i=\alpha_i\circ \Psi(f)$ 
and $\Psi(f)\circ\beta_i=\beta_i\circ \Psi(f)$. From the definition of 
a natural transformation we have $\theta_{i+1}^{\mathrm{out}}\theta_{i+1}^{\mathrm{on}}(g)\circ 
\alpha_i=\alpha_i\circ g$ for any morphism $g$, in particular
$\theta_{i+1}^{\mathrm{out}}\theta_{i+1}^{\mathrm{on}}
(\theta_{i}^{\mathrm{out}}(g))\circ \alpha_i=\alpha_i\circ\theta_{i}^{\mathrm{out}}(g)$. 
From the definition of $\Psi$ it follows that
\begin{displaymath}
\Psi_{i+1}(f)\circ \alpha_i=\alpha_i\circ\Psi_i(f)
\end{displaymath}
and then $\Psi(f)\circ \alpha_i=\alpha_i\circ\Psi(f)$.  
Analogously one obtains $\Psi(f)\circ\beta_i=\beta_i\circ \Psi(f)$.
\end{proof}

\begin{proof}[Proof of Theorem~\ref{thm21}.]
By Theorem~\ref{thmmain}\eqref{thmmain.2}, Proposition~\ref{freeness} 
and Lemma~\ref{Rcentral}, the algebra $\hat{D}_{\la,\xi}$ has the 
structure of a $\mC[[x_1,x_2,\dots,x_n]]$-algebra with the basis 
as described in Proposition~\ref{freeness}\eqref{freeness.2}.
Restricting the above arguments to $\mathcal{C}_{\lambda,\xi}$
and using Theorem~\ref{thmmain}\eqref{thmmain.1},
we get that the algebra ${D}_{\la,\xi}$ has the structure of a
$\mC[[x]]$-algebra with the basis as described in 
Proposition~\ref{freeness}\eqref{freeness.2}. Hence there are  
obvious isomorphisms $\hat{D}_{\la,\xi}\cong A^{n-1}[[x_1,x_2,\dots,x_n]]$
and ${D}_{\la,\xi}\cong A^{n-1}[[x]]$ of 
$\mC[[x_1,x_2,\dots,x_n]]$- and $\mC[[x]]$-modules, respectively, 
sending $\alpha_i$ to $a_i$ and $\beta_i$ to $b_i$. 

Reducing modulo $\mathfrak{m}$ we get that the algebra 
$\hat{D}_{\la,\xi}^{(1)}\cong {D}_{\la,\xi}^{(1)}$ is generated by
the images of $\alpha_i$ and $\beta_i$ (since $n>2$). Thanks to \eqref{nohom}, 
we have the relations $\alpha_{i+1}\alpha_i=0$ and  $\beta_{i-1}\beta_i=0$. 

The compositions $\beta_{i}\alpha_i$ and  $\alpha_i\beta_{i}$ 
send the top of $\theta_i^{\mathrm{out}}L_i$ to the socle. This implies
$\alpha_i\beta_{i}\alpha_i=0$ and $\beta_{i}\alpha_i\beta_i=0$ for
all $i$. Therefore in case $n=3$ we get ${D}_{\la,\xi}^{(1)}\cong A^{2}$.

Since $\theta_i^{\mathrm{out}}L_i$ has simple socle, for
$n>3$ we obtain $\beta_{i}\alpha_i=c_i\alpha_i\beta_{i}$ for 
some $c_i\in\mathbb{C}$. That $c_i\not=0$ can easily be verified by a direct calculation. Rescaling $\alpha_{i-1}$, if 
necessary, we may assume $c_i=1$ for all $i$. It follows
that ${D}_{\la,\xi}^{(1)}\cong A^{n-1}$ for all $n$, which
completes the proof of Theorem~\ref{thm21}.
\end{proof}

\section{Explicit description of deformations}\label{s4}

\subsection{Deformations and Hochschild cohomology}\label{s4.1}

We use the notation from Subsection~\ref{s3.1}.
Identify $\mathbb{C}_{U}$ with $\mathbb{C}[[u_1,u_2,\dots,u_m]]$,
where $u_1,u_2,\dots,u_m$ is a basis of $U$. Then a deformation
$A_U$ of $A$ over $U$ is nothing else than the vector space
$A[[u_1, u_2,\ldots u_m]]$ together with a  $k[[\mathbf{u}]]$-linear
associative star product
\begin{equation}\label{eq33}
a\star b=\sum\mu_{\bf d}(a,b)\mathbf{u}^{\mathbf{d}},
\end{equation}
where the sum runs over all multi-indices
${\bf d}=(d_1,d_2,\ldots, d_m)\in\mathbb{Z}_{+}^m$, 
$\mathbf{u}=(u_1,u_2,\dots,u_m)$, 
$\mathbf{u}^{\mathbf{d}}=u_1^{d_1}u_2^{d_2}\cdots u_m^{d_m}$, 
$\mu_{\mathbf{d}}:A\otimes A\to A$ is a linear map and
$\mu_{\bf 0}(a,b)=ab$ for all $a,b\in A$.  In particular, if
$m=1$ and $u=u_1$, then
\begin{equation}\label{eq34}
a\star b=ab+\mu_1(a,b)u+\mu_2(a,b)u^2+\cdots.
\end{equation}

A classical result of Gerstenhaber (\cite{Ge}) says that infinitesimal
one-parameter deformations are classified by the second Hochschild
cohomology $\mathbf{HH}^2(A,A)$ of $A$ with values in $A$ in
the sense that the associativity of the star product implies that
$\mu_1:A\otimes A\rightarrow A$ is always a $2$-cocycle and the
isomorphism classes (of deformations) are exactly given by the 
coboundaries. The analogous statement for $m$-parameter deformations 
reads as follows (see \cite[Theorem 1.1.5]{BeGi}):

\begin{proposition}\label{prop41}
Isomorphism classes of infinitesimal $m$-parameter deformations of
$A$ are in bijection with the space
\begin{displaymath}
\mathrm{Hom}\big(
\mathrm{Hom}(\mathfrak{m}/\mathfrak{m}^2,\mathbb{C}),
\mathbf{HH}^2(A,A)\big).
\end{displaymath}
\end{proposition}

\begin{proof}
Let $B=A_U$ be an $m$-parameter deformation of $A$. Let $K$ be the kernel of
the multiplication map $B\otimes B\rightarrow B$. Then we have a short
exact sequence
\begin{displaymath}
0\rightarrow K\rightarrow B\otimes B\rightarrow B\rightarrow 0.
\end{displaymath}
Note that the kernel $K$ is a free right $B$-module, hence we may
reduce modulo $\mathfrak{m}$ from the right hand side and obtain a
short exact sequence
\begin{displaymath}
0\rightarrow K\otimes_{\mathbb{C}_U}\mathbb{C}_U/\mathfrak{m}\rightarrow
B\otimes B\otimes_{\mathbb{C}_U}\mathbb{C}_U/
\mathfrak{m}\rightarrow
B\otimes_{\mathbb{C}_U}\mathbb{C}_U/\mathfrak{m}\rightarrow 0.
\end{displaymath}
By reducing the latter modulo $\mathfrak{m}$ from the left hand side
we obtain  an exact sequence
\begin{equation}
\label{Ext2}
0\rightarrow A\otimes\mathfrak{m}/\mathfrak{m}^2
\rightarrow \mathbb{C}_U/\mathfrak{m}\otimes_{\mathbb{C}_U}K
\otimes_{\mathbb{C}_U}\mathbb{C}_U/\mathfrak{m}
\rightarrow A\otimes A
\rightarrow A\rightarrow 0,
\end{equation}
using the identification
$\mathbb{C}_U/\mathfrak{m}\otimes_{\mathbb{C}_U}A\otimes
A\otimes_{\mathbb{C}_U}\mathbb{C}_U/\mathfrak{m}=A\otimes A$ and
the identification
$\mathbb{C}_U/\mathfrak{m}\otimes_{\mathbb{C}_U}A\otimes_{\mathbb{C}_U}\mathbb{C
}_U/\mathfrak{m}=A$, the fact that $A\otimes A$ is a free left $A$-module,
and finally the isomorphism
\begin{displaymath}
\mathrm{Tor}_1(\mathbb{C}_U/\mathfrak{m}, A)\cong
A\otimes\mathrm{Tor}_1(\mathbb{C}_U/\mathfrak{m},\mathbb{C}_U/\mathfrak{m})=
A\otimes\mathfrak{m}/\mathfrak{m}^2.
\end{displaymath}
In particular, \eqref{Ext2} defines an element in
\begin{displaymath}
\mathrm{Ext}^2_{A-A}(A,A\otimes\mathfrak{m}/\mathfrak{m}^2)=
\mathrm{Hom}\big(\mathrm{Hom}(\mathfrak{m}/\mathfrak{m}^2,\mathbb{C}),
\mathbf{HH}^2(A,A)\big).
\end{displaymath}
One can check that this defines the required isomorphism.
\end{proof}

\subsection{Flat deformations associated with associative 
Hochschild $2$-cocycles}\label{s4.2}

Consider the natural partial order $\leq$ on $\mathbb{Z}_+^m$, defined
as follows: $(d'_1,d'_2,\dots,d'_m)\leq (d_1,d_2,\dots,d_m)$ provided
that $d'_i\leq d_i$ for all $i$. Note that for 
$\mathbf{d}',\mathbf{d}\in \mathbb{Z}_+^m$ we have
$\mathbf{d}'\leq \mathbf{d}$ if and only if 
$\mathbf{d}-\mathbf{d}'\in \mathbb{Z}_+^m$. Recall that 
$\{\varepsilon_i|i=1,2,\dots,m\}$ is the standard
basis of $\mathbb{Z}_+^m$. Set $\mathbf{0}=(0,0,\dots,0)$.

In the notation of \eqref{eq33}, 
the associativity equation $(a\star b)\star c=a\star (b\star c)$ 
reduces to the following system of equations, where 
$\mathbf{d}\in\mathbb{Z}_+^m$:
\begin{equation}\label{eq35}
\sum_{\mathbf{d}'\leq \mathbf{d}} 
\mu_{\mathbf{d}'}(\mu_{\mathbf{d}-\mathbf{d}'}(a,b),c)
=\sum_{\mathbf{d}'\leq \mathbf{d}}
\mu_{\mathbf{d}'}(a,\mu_{\mathbf{d}-\mathbf{d}'}(b,c)).
\end{equation}
In particular, in the notation of  \eqref{eq34} this reduces to the
following:
\begin{multline}\label{eq35n}
\mu_k(a,b)c+\mu_k(ab,c)-a\mu_k(b,c)-\mu_k(a,bc)=\\=
\sum_{i=1}^{k-1}\big( \mu_i(a,\mu_{k-1}(b,c))-\mu_i(\mu_{k-1}(a,b),c)\big).
\end{multline}
The right hand side of \eqref{eq35n} is a Hochschild $3$-cocycle and
the equation is solvable if and only if it is a coboundary. Hence
obstructions to extend infinitesimal deformations to flat deformations
are given by $\mathbf{HH}^3(A,A)$. We will need the following
easy observation:

\begin{proposition}\label{lem45}
Let $\nu$ be a Hochschild $2$-cocycle of $A$. For any 
$\mathbf{d}\in\mathbb{Z}_+^m\setminus\{\mathbf{0}\}$
choose $c_{\mathbf{d}}\in\mathbb{C}$.
Assume that $\nu$ is associative, that is $\nu(\nu(a,b),c)=\nu(a,\nu(b,c))$.
Then setting $\mu_{\mathbf{d}}=c_{\mathbf{d}}\nu$ for all $\mathbf{d}
\in\mathbb{Z}_+^m\setminus\{\mathbf{0}\}$
defines a flat deformation of $A$ over $\mathbb{C}[[u_1,u_2,\dots,u_m]]$.
\end{proposition}

\begin{proof}
For any $\mathbf{0}<\mathbf{d}'<\mathbf{d}$ we have
\begin{multline*}
\mu_{\mathbf{d}'}(\mu_{\mathbf{d}-\mathbf{d}'}(a,b),c)=
c_{\mathbf{d}'}c_{\mathbf{d}-\mathbf{d}'}\nu(\nu(a,b),c)=\\
c_{\mathbf{d}'}c_{\mathbf{d}-\mathbf{d}'}\nu(a,\nu(b,c))=
\mu_{\mathbf{d}'}(a,\mu_{\mathbf{d}-\mathbf{d}'}(b,c))
\end{multline*}
because of the associativity of $\nu$. Taking into account that
$\nu$ is a Hochschild $2$-cocycle, we also have
\begin{multline*}
\mu_{\mathbf{d}}(a,b)c+\mu_{\mathbf{d}}(ab,c)-
a\mu_{\mathbf{d}}(b,c)-\mu_{\mathbf{d}}(a,bc)=\\
c_{\mathbf{d}}(\nu_{\mathbf{d}}(a,b)c+\nu_{\mathbf{d}}(ab,c)-
a\nu_{\mathbf{d}}(b,c)-\nu_{\mathbf{d}}(a,bc))=
0.
\end{multline*}
This implies that all equations in the system \eqref{eq35} are, 
in fact, identities. The claim follows.
\end{proof}

\begin{corollary}\label{cordef}
Assume that $\mathbf{HH}^2(A,A)$ is one-dimensional and that there exists
a nontrivial Hochschild $2$-cocycle $\nu$, which is associative
(as in Proposition~\ref{lem45}). Then any two $m$-parameter flat 
deformations $B_i$, $i=1,2$, of $A$, which are nontrivial
deformations when reduced modulo $\mathfrak{m}^2$, are isomorphic 
as associative algebras (but not necessarily as deformations).
\end{corollary}

\begin{proof}
Choose $c_{\mathbf{d}}=0$ for all
$\mathbf{d}\in\mathbb{Z}_+^m\setminus\{\mathbf{0}\}$
except of $c_{\varepsilon_1}=1$. Then we have the corresponding flat 
deformation $B$ of $A$, given by  Proposition~\ref{lem45}.
To prove the corollary we may assume $B=B_2$.

Assume that $B_1$ is a deformation of $A$ over
$\mathbb{C}[[h_1,h_2,\dots,h_m]]$ with the star product given by
\eqref{eq33}. From Proposition~\ref{prop41}
it follows that up to isomorphism of deformations we may assume that 
$\mu_{\varepsilon_i}=c'_{\varepsilon_i}\nu$ for some
$c'_{\varepsilon_i}\in\mathbb{C}$ for all $i$. As
$B_1$ is nontrivial, when reduced modulo $\mathfrak{m}^2$, at least one
of $c'_{\varepsilon_i}$ is nonzero. 

Now we prove, by induction with respect to the order $\leq$, that 
$\mu_{\mathbf{d}}=c'_{\mathbf{d}}\nu$ for some 
$c'_{\mathbf{d}}\in\mathbb{C}$ for all
$\mathbf{d}\in\mathbb{Z}_+^m\setminus\{\mathbf{0}\}$. The basis
of the induction is proved in the previous paragraph. To prove the
induction step we take some 
$\mathbf{d}\in\mathbb{Z}_+^m\setminus\{\mathbf{0}\}$. Using the
inductive assumption, the corresponding equation from 
\eqref{eq35} reduces to 
\begin{displaymath}
\mu_{\mathbf{d}}(a,b)c+\mu_{\mathbf{d}}(ab,c)-  
a\mu_{\mathbf{d}}(b,c)-\mu_{\mathbf{d}}(a,bc)=0.
\end{displaymath}
This means that $\mu_{\mathbf{d}}$ must be a Hochschild $2$-cocycle.
As $\mathbf{HH}^2(A,A)$ is one-dimensional, up to isomorphism of
deformations we thus have $\mu_{\mathbf{d}}=c'_{\mathbf{d}}\nu$ for some 
$c'_{\mathbf{d}}\in\mathbb{C}$.

It follows that the star product for $B_1$ looks as follows:
\begin{displaymath}
a\star b= ab+\sum_{\mathbf{d}\neq\mathbf{0}} 
c'_{\mathbf{d}}\nu_{\mathbf{d}}(a,b)
\mathbf{h}^{\mathbf{d}}=ab+\nu_{\mathbf{d}}(a,b)
\left(\sum_{\mathbf{d}\neq\mathbf{0}} 
c'_{\mathbf{d}}\mathbf{h}^{\mathbf{d}}\right).
\end{displaymath}

Without loss of generality we may assume that $c'_{\varepsilon_1}$ 
is nonzero.  Then there is an isomorphism 
from $\mathbb{C}[[u_1,u_2,\dots,u_m]]$ 
to $\mathbb{C}[[h_1,h_2,\dots,h_m]]$ which maps $u_1$ to 
$\sum_{\mathbf{d}\neq\mathbf{0}}c'_{\mathbf{d}}\mathbf{h}^\mathbf{d}$
and $u_i$ to $h_i$ for all $i>1$. This induces a ring isomorphism
from ${B}_2$ to $B_1$. The statement follows.
\end{proof}

\subsection{Hochschild cohomology of $A^k$}\label{s4.3}

The even Hochschild cohomology of the algebras $A^k$ was first described in \cite{Ho1}. Thanks to \cite{Rickard}, \cite{EH}, \cite{Ho1}, there 
exists moreover an
explicit description of the full Hochschild cohomology ring. We will not need these detailed descriptions and just recall the result about the
dimensions:

\begin{proposition}\label{HolmandCo}
We have
\begin{displaymath}
\dim\mathbf{HH}^i(A^{k})=
\begin{cases}
k+1, & i=0;\\
1, & i>0.
\end{cases}
\end{displaymath}
\end{proposition}

\begin{proof}
Thanks to \cite[Theorem 4.2]{Rickard}, every Brauer tree algebra with
$k$ simple modules and exceptional multiplicity $m$ is derived equivalent
to a path algebra $B^m_k$ of a quiver which is a single oriented cycle of
length $k$ modulo the ideal generated by all pathes of length $mk+1$
(the latter is again a Brauer tree algebra). The algebras $A^k$ are
such Brauer tree algebras with $m=1$. Since Hochschild cohomology is
an invariant of the derived category, it is enough to determine the
Hochschild cohomology of $B_k^1$. Then the result is a special case
of \cite[Theorem 8.1]{Ho1} and \cite[Theorem 5.19]{EH}.
\end{proof}

In particular, we have $\dim\mathbf{HH}^2(A^{k})=1$. The corresponding
nontrivial Hochschild $2$-cocycle is given by the following:

\begin{lemma}\label{hochcoc}
In the basis $\mathbf{B}_k$ of $A^k$, the assignment
\begin{equation}\label{mu}
\begin{array}{ccc}
\mu:\quad A^{k}\otimes A^{k}&\longrightarrow&A^{k}\\
a_s\otimes b_s&\longmapsto&(-1)^{s+1}e_{s+1},\\
b_1\otimes a_1&\longmapsto&e_1,\\
a_sb_s\otimes a_sb_s&\longmapsto&(-1)^s a_sb_s,\\
a_s\otimes a_{s-1}b_{s-1}&\longmapsto&(-1)^{s-1}a_s\\
a_{s-1}b_{s-1}\otimes b_{s}&\longmapsto&(-1)^{s-1}b_s,
\end{array}
\end{equation}
extended by zero to all other basis vectors, is a 
nontrivial associative Hochschild $2$-cocycle.
\end{lemma}

\begin{proof}
That $\mu$ is an associative Hochschild $2$-cocycle is checked
by a straightforward computation. The element $e_1$ does not belong 
to the radical of $A^{k}$, while both $b_1$ and $a_1$ do. Hence 
$b_1\otimes a_1\mapsto e_1$ is not possible for any coboundary, which 
implies that $\mu$ is nontrivial. The claim follows.
\end{proof}

\subsection{One-parameter deformations of $A^k$}\label{s4.4}

Let us start with an explicit example, illustrating
Proposition~\ref{HolmandCo} and Corollary~\ref{cordef}.

\begin{example}\label{dualnumbers}
{\rm
Consider the algebra $A^1=\mathbb{C}[X]/(X^2)$ and the bimodule
$R:=A^1\otimes A^1$.
Let $f,g:R\rightarrow A^1$ be the $A^1$-bimodule maps given by
$f(1\otimes 1)=X\otimes 1-1\otimes X$ and
$g(1\otimes 1)=X\otimes 1+1\otimes X$. Then
\begin{displaymath}
\cdots\stackrel{g}\rightarrow R\stackrel{f}\rightarrow R\stackrel{g}{\rightarrow} R\stackrel{f}{\rightarrow} R\stackrel{\mathrm{mult}}{\rightarrow} A^1_0\rightarrow 0
\end{displaymath}
is a free $R$-module resolution of $A^1$. A direct computation gives
$\mathbf{HH}^0(A^1,A^1)=A^1$ (as $A^1$ is commutative) and
$\mathbf{HH}^i(A^1,A^1)\cong \mathbb{C}$ for all $i>0$
(as claimed by Proposition~\ref{HolmandCo}).
In particular, $\mathbf{HH}^2(A^1,A^1)=\mathbb{C}$, hence
there is a unique nontrivial infinitesimal one-parameter deformation
for which we can take the corresponding nontrivial associative
Hochschild $2$-cocycle $\mu_1$ as follows: $\mu_1(a,b)=0$ 
for $(a,b)\in\{(1,1),(1,X),(X,1)\}$ and $\mu_1(X,X)=1$. Hence,
by Corollary~\ref{cordef} we have a unique, up to isomorphism,
associative algebra, which is an $m$-parameter deformation of $A^1$, 
nontrivial when reduced modulo $\mathfrak{m}^2$. If $m=1$, this is realized
as follows: consider $A^1[[t]]$ with $X\star X=t$ and therefore isomorphic to
$\mathbb{C}[[u]]$ by sending $X$ to $u$. If $m>1$, we similarly
get the polynomial algebra $\mathbb{C}[[u_1,u_2,\dots,u_m]]$.
}
\end{example}

This example can be generalized to all $A^k$ as follows: For $k>1$
consider the quotient $\hat{B}^k$ of the path algebra of the following 
quiver (in case $k$ is even or odd, respectively):
\begin{displaymath}
\xymatrix{
\ar@(dl,ul)^{y_1}{1}\ar@/^/[r]^{x_1}&\ar@/^/[l]^{x_2}
{2}\ar@/^/[r]^{y_2}&\ar@/^/[l]^{y_3}{3}\ar@/^/[r]^{x_3}
&\cdots\ar@/^/[l]^{x_4}\ar@/^/[r]^{x_{k-1}}&\ar@/^/[l]^{x_{k}}
{k}\ar@(ur,dr)^{y_{k}}\\ \\
\ar@(dl,ul)^{y_1}{1}\ar@/^/[r]^{x_1}&\ar@/^/[l]^{x_2}
{2}\ar@/^/[r]^{y_2}&\ar@/^/[l]^{y_3}{3}\ar@/^/[r]^{x_3}
&\ar@/^/[l]^{x_4} \cdots
\ar@/^/[r]^{y_{k-1}}
&\ar@/^/[l]^{y_{k}}{k}\ar@(ur,dr)^{x_{k}}\\
}
\end{displaymath}
modulo the relations $x_iy_j=0=y_jx_i$ whenever the expression makes
sense. Let $B^k$ denote the completion of $\hat{B}^k$ with respect to the
ideal generated by all arrows.

Consider also the deformation $A^k[[t]]$ of $A^k$ in which the product 
$\star$ is given by $x\star y=xy+\mu(x,y)t$, where $\mu$ is as
in Lemma~\ref{hochcoc}. Associativity of this product follows 
from Proposition~\ref{lem45}. Obviously, $A^k[[t]]$ is nontrivial when 
reduced modulo $\mathfrak{m}^2$.

\begin{theorem}\label{thm29}
\begin{enumerate}[(i)]
\item\label{thm29.1} Putting the $x$'s and $y$'s in degree one, turns 
the algebra $\hat{B}^{k}$ into a Koszul algebra.
\item\label{thm29.2} The algebra $B^k$ has the structure of a 
one-parameter flat deformation of $A^k$, isomorphic, as a deformation,
to the deformation $A^k[[t]]$ described above. 
\end{enumerate}
\end{theorem}

\begin{proof}
The algebra $\hat{B}^k$ is by definition quadratic with monomial relations,
hence Koszul, see \cite[Corollary II.4.3]{PP} (it is in fact straightforward 
to write down projective resolutions of simple modules and see that they 
are linear). This proves \eqref{thm29.1}. To prove the second statement
we will have to check several things.

For every point $i$ of the quiver set $t_i=x(i)-y(i)$, where $x(i)$
is the shortest $x$-loop starting in $i$ and $y(i)$ is the shortest
$y$-loop starting in $i$ (for example, $y(1)=y_1$, $x(1)=x_2x_1$,
$x(2)=x_1x_2$, $y(2)=y_2$ if $k=2$ and $y(2)=y_3y_2$ if $k>2$ and so on).
Define $t(k)$ to be the sum of all $t_i$'s. The element $t(k)$
is a sum of loops and hence commutes with all idempotents of the path
algebra. As all $x$-monomials of $t(k)$ have coefficient $1$ and the
product of any $x$ and any $y$ is zero, it follows that
$t(k)$ commutes with all $x_i$. As all $y$-monomials of $t(k)$ have
coefficient $-1$ and the product of any $x$ and any $y$ is zero, it
follows that $t(k)$ commutes with all $y_j$. Hence $t(k)$ is a
central element of $\hat{B}^k$.

Let $I$ denote the ideal of $\hat{B}^k$, generated by all arrows
and $J$ denote the ideal of $\hat{B}^k$, generated by $t(k)$.
Then $J\subset I$. At the same time,
$I^4$ is generated by elements of the form
$x_{i}x_{i\pm 1}x_{i}x_{i\pm 1}$ and
$y_{i}y_{i\pm 1}y_{i}y_{i\pm 1}$, which both belong to
$J^2$. Hence $I^4\subset J^2$, which implies that the completions of
$\hat{B}^k$ with respect to $I$ and $J$ coincide, and hence both are
isomorphic to $B^k$.

Next we claim that $X:=\hat{B}^k/J\cong A^k$. Indeed, using the
relation $x_iy_j=0=y_jx_i$ it is easy to see that there is a
unique algebra homomorphism $\varphi:\hat{B}^k\to A^k$, which maps
$y_1$ to $b_1a_1$, $x_1$ to $a_1$, $x_2$ to $b_1$,
$y_2$ to $a_2$, $y_3$ to $b_3$ and so on (i.e. all non-loop arrows are
mapped to the corresponding non-loop arrows, and loops at the
ends are mapped to length two loops at the ends). From the definition
we have that $\varphi$ is surjective and maps $t(k)$ to $0$. Therefore
it induces a surjective homomorphism
$\overline{\varphi}:\hat{B}^k/J\tto A^k$. At the same time
it is easy to see that the images of non-loop arrows in $\hat{B}^k/J$
satisfy the defining relations of $A^k$. Therefore
$\overline{\varphi}$ is injective and hence an isomorphism.

It is easy to see that the algebra $\hat{B}^k$ has a $\mathbb{C}$-basis
$\hat{\mathbf{B}}_k$, which consists of all trivial paths, all monomials
in $x_i$'s and all monomials in $y_j$'s which make sense.
We claim that $\hat{B}^k$ is free over $\mathbb{C}[t(k)]$.
The map $\varphi$ maps elements from $\hat{\mathbf{B}}_k$ to elements from
${\mathbf{B}}_k$ or zero. For every element from $\mathbf{B}_k$ fix
its tautological preimage under $\varphi$ (for example $e_i$ is the
preimage of $e_i$, $x_1$ is the preimage of $a_1$, $x_2x_1$ is the
preimage of $b_1a_2$ and so on). Call the latter set
$\underline{\hat{\mathbf{B}}}_k$. There is an obvious bijection between
the elements in $\hat{\mathbf{B}}_k$ and elements $zt(k)^i$,
where $z\in \underline{\hat{\mathbf{B}}}_k$ and $i\in\mathbb{Z}_+$. It follows
that $\underline{\hat{\mathbf{B}}}_k$ is a free $\mathbb{C}[t(k)]$-basis
of $\hat{B}^k$. The latter yields that $B^k$ is free over
$\mathbb{C}[[t(k)]]$  and hence is a flat deformation of $A^k$.

It is easy to check that $x_1x_2x_1\neq 0$ in $B^k/J^2$, which yields
that this deformation is nontrivial when reduced modulo $J^2$.
Therefore \eqref{thm29.2} follows from Corollary~\ref{cordef},
completing the proof of the theorem.
\end{proof}

\begin{remark}\label{remnn123}
{\rm 
One easily checks that the map
\begin{displaymath}
\begin{array}{ccc}
\Psi: A^{m}[[t]]&\longrightarrow&B^{m},\\
e_i&\mapsto&e_i,\\
a_s&\mapsto&
\begin{cases}
x_s &\text{if $s$ odd}\\
y_s &\text{if $s$ even}
\end{cases}\\
b_s&\mapsto&
\begin{cases}
x_{s+1} &\text{if $s$ odd,}\\
y_{s+1} &\text{if $s$ even,}
\end{cases}\\
b_1a_1&\mapsto& y_1,\\
a_{k-1}b_{k-1}&\mapsto&
\begin{cases}
y_k, &\text{if $k$ even,}\\
x_k, &\text{if $k$ odd,}
\end{cases},\\
a_sb_s&\mapsto&
\begin{cases}
x_{s}x_{s+1}, &\text{if $s$ even,}\\
y_{s}y_{s+1}, &\text{if $t$ odd,}
\end{cases}\\
t&\mapsto&t(k)
\end{array}
\end{displaymath}
for $i$ and $s$ for which the expression makes sense, is an isomorphism
of deformations. 
}
\end{remark}

The algebras $\hat{B}^k$ from Theorem~\ref{thm29} are studied in
\cite{GP}, \cite{Kh}. They belong to the class of special biserial algebras
(\cite{WW}), in particular, they are tame.  A complete description
of indecomposable modules over these algebras can be found in
\cite{Kh}, \cite{WW}, \cite{GS}. The algebra $\mathbb{C}[[x,y]]$ is 
a classical wild algebra. Therefore from 
Theorem~\ref{thmmain}\eqref{thmmain.2} and
Theorem~\ref{thm21} it follows that 
all nonzero categories $\hat{\mathcal{C}}_{\lambda,\xi}$ are wild.

\begin{remark}\label{rem29}
{\rm
By Theorem~\ref{thm29}, the algebra $\hat{B}^k$ is Koszul,  as well as 
$A^1$. This is however not the case for $A^k$ if $k>1$ (as can easily be 
seen by writing down explicit projective resolutions, for $k=2$ the 
algebra is not even quadratic).
}
\end{remark}

\subsection{One-parameter graded deformations of $A^k$}\label{s4.5}

The algebra $A^{k}$ can be equipped with two different natural
nonnegative $\mathbb{Z}$-gradings $A^{k}=\oplus_{j\geq 0}A^{k}(j)$
by either putting the $a_i$'s and the $b_i$'s all in degree one or by
putting the $a_i$'s in degree one and the $b_i$'s in degree zero,
respectively. In the first case we get a positive grading in the sense
of \cite{MOS}, which is natural in the study of Koszul algebras.

\begin{remark}\label{rem30}
{\rm
In \cite{BLPPW} universal Koszul deformations were studied.
This construction leads to different deformations than the one
obtained in this paper as universal Koszul deformations behave badly
with respect to taking the centralizer subalgebras. For example, as 
mentioned earlier,  the algebra $\mathbb{C}[X]/(X^2)\cong A^1$ is the 
centralizer subalgebra of the Koszul quasi-hereditary algebra $\tilde{A}^1$ 
(the latter one is graded in the usual way by path length). 
Let $\tilde{A}^1[t]$ be the universal Koszul deformation in the sense 
of \cite{BLPPW}. Then one might hope that the corresponding 
centralizer subalgebra in $\tilde{A}^1[t]$ is isomorphic to the 
deformation of $A^1$ described in Example~\ref{dualnumbers}. 
This is however not the case since the first one is isomorphic
to $\mathbb{C}[X]\otimes_{\mathbb{C}[X^2]}\mathbb{C}[X]$,
whereas the second one is isomorphic to $\mathbb{C}[X]$.
Analogous phenomenon appears for all $A^k$.
}
\end{remark}

\begin{remark}\label{rem30.1}
{\rm
The $0$-th Hochschild cohomology of $A^k$, i.e. the center of the algebra $A^k$, or, more generally, of generalized Khovanov algebras has a geometric interpretation via the cohomology ring of Springer fibres (\cite{Br}, \cite{St2}). In \cite{St2}, the Koszul deformation (in the sense 
of Remark~\ref{rem30}) was used as an essential tool to obtain the result. 
The center of this deformed algebra is in fact the $T$-equivariant 
cohomology ring (\cite{St2}, \cite{GM}), for a one-dimensional torus $T$.
}
\end{remark}

Despite of Remark~\ref{rem29}, the algebra $B^k$ can be understood using
the graded picture. The algebra $\hat{B}^k$ is graded in the natural way
by putting all non-loop arrows in degree one and loops in degree two
(in case $m=1$ the only loop is put in degree one). The following was
established during the proof of Theorem~\ref{thm29}: The algebra 
$\hat{B}^k$ is free over the graded central subalgebra $\mathbb{C}[t(k)]$ 
and the quotient of $\hat{B}^k$ over the homogeneous ideal generated by 
$t(k)$ is isomorphic to $A^k$ as a graded algebra. This means that
$\hat{B}^k$ is a graded flat one-parameter deformation of $A^k$.
The algebra $B^k$ is the completion of the positively graded algebra 
$\hat{B}^k=\oplus_{j\geq 0}\hat{B}^k(j)$ with respect to the 
graded radical, that is the ideal $\oplus_{j>0}\hat{B}^k(j)$. 
The graded version of the deformation theory is described in \cite{BrGa}.
Using these results we obtain:

\begin{corollary}\label{cor48}
Let $k\in\mathbb{N}$. Consider $A^k$ as a graded algebra by assuming
that all arrows have degree one (except $k=1$ when the only loop is
assumed to have degree two). Then $\hat{B}^k$ is the 
unique (up to rescaling of the deformation parameter)
nontrivial graded flat one-parameter deformation of $A^k$.
\end{corollary}

\begin{proof}
The algebra $\hat{B}^k$ has the structure of a nontrivial 
graded flat one-parameter deformation of $A^k$ given by Theorem~\ref{thm29}. 
In particular, $A^k$ has a nontrivial
graded infinitesimal deformation, which is unique up to rescaling of the deformation parameter as $\mathbf{HH}^2(A^k,A^k)$ is one-dimensional.
The existence of $\hat{B}^k$ says that this deformation extends to a 
nontrivial flat graded deformation of $A^k$ by Theorem~\ref{thm29}. By
\cite[Proposition~1.5(c)]{BrGa}, the freedom of extending our infinitesimal
deformation to a flat deformation is given by the graded pieces of
$\mathbf{HH}^2(A^k,A^k)$ in higher degrees. These spaces are however
zero as $\mathbf{HH}^2(A^k,A^k)$ is one-dimensional and hence
is concentrated in the degree which produces the non-trivial
infinitesimal deformation.
\end{proof}

Consider now the second grading on $A^k$ (obtained by putting the
$a_i$'s in degree $1$ and the $b_i$'s in degree zero, respectively).
This grading is no longer positive in the sense of \cite{MOS}.
To distinguish the obtained graded algebra from the previous
graded algebra, we denote it by $\tilde{A}^{k}$. Similarly we define
a grading on $\hat{B}^k$ by putting all right arrows and end
loops in degree one and all left arrows in degree zero. Let $\tilde{B}^k$ be the resulting algebra.

\begin{proposition}\label{prop51}
Let $k\in\{2,3,\dots\}$. Then the algebra $\tilde{B}^k$ is the unique
(up to rescaling of the deformation parameter) non-trivial graded
flat one-parameter deformation of $\tilde{A}^{k}$.
\end{proposition}

\begin{proof}
Mutatis mutandis that of Corollary~\ref{cor48}.
\end{proof}

\begin{remark}\label{rem49}
{\rm
The graded algebra $\tilde{A}^{k}$ is a subalgebra of the algebra studied in
\cite{KhSe} in the context of Floer cohomology of Lagrangian
intersections in the symplectic manifolds which are Milnor fibres of
simple singularities of type $A_{k}$.
}
\end{remark}

\subsection{Proof of Theorem~\ref{thmmain}\eqref{thmmain.3}
and \eqref{thmmain.4}}\label{s4.6}

Let $\lambda$ be as in case~\eqref{case3} and $\xi$ be such that
$\hat{\mathcal{C}}_{\lambda,\xi}$ is nontrivial.  We will use the 
notation from Subsection~\ref{s3.3}. Theorem~\ref{thm21} gives the 
algebra $\hat{D}_{\lambda,\xi}$ the structure of a deformation of 
$A^{n-1}$ over $\mathbb{C}[[x_1,x_2,\dots,x_n]]$ and the algebra 
${D}_{\lambda,\xi}$ the structure of a deformation of $A^{n-1}$ over 
$\mathbb{C}[[x]]$.

\begin{lemma}\label{lem092}
There is a uniserial module $M\in {D}_{\lambda,\xi}$ of Loewy length
four such that the layers of the radical filtration of $M$ are 
$\hat{L}_1$, $\hat{L}_2$, $\hat{L}_1$ and $\hat{L}_2$.
\end{lemma}

\begin{proof}
This follows from \cite[Lemma~6.4]{GS}. Alternatively one can construct 
this module using the Gelfand-Zetlin approach, described in \cite{Maz}.
Let us briefly outline the latter. Let $V$ be a two dimensional vector
space and $X$ be a nonzero nilpotent operator on $V$. 

As explained in Subsection~\ref{s1.6}, the element $\lambda$ corresponds
to some  $(m_1,m_2,m_3,\dots,m_{n})\in\mathbb{C}^n$ such that 
$m_2-m_1$, $m_1-m_3$, $m_3-m_4$,\dots, $m_{n-1}-m_n\in \mathbb{N}$.
By Proposition~\ref{prop3}, we may arbitrarily change $\xi$, so we
choose $x_1,x_2,\dots,x_{n-1}\in \mathbb{C}$ such 
that  $m_1-x_i\not\in\mathbb{Z}$ and $x_i-x_j\not\in\mathbb{Z}$ for all 
$i,j$. Consider the set $T(\mathbf{m},\mathbf{x})$ and the
corresponding module $M(\mathbf{m},\mathbf{x})$ as 
defined in Subsection~\ref{s1.6}. Note that we are in case \eqref{case3}
and $m_2-m_1,m_1-m_3\in \mathbb{N}$. Hence, by construction, the 
module $M(\mathbf{m},\mathbf{x})$ is a uniserial module of length 
two with subquotients of the radical filtration being 
$\hat{L}_2$, $\hat{L}_1$. Now we would like to construct a
nontrivial selfextension of this module module
using $X$.

For every $t\in T(\mathbf{m},\mathbf{x})$ let $V_t$ denote a copy 
of $V$. Consider the vector space 
$M=\oplus_{t\in T(\mathbf{m},\mathbf{x})}V_t$. For $i\in\{1,2,\dots,n-1\}$
define on $M$ the action of $e_{i,i+1}$ and $e_{i+1,i}$ by the
classical Gelfand-Zetlin formulae (\cite{DFO}, \cite{Maz}) in which we 
treat all entries of our tableaux as the corresponding scalar operators 
on $V$ with the exception of $m_1$, which should be understood as 
$m_1\mathrm{Id}_V+X$, and $y_i$, $i=1,2,\dots,n-1$, which should 
be understood as $y_i\mathrm{Id}_V+X$. Similarly to the proof 
of Lemma~\ref{lem7} one 
shows that this defines on $M$ the structure of a $\mathfrak{g}$-module 
for any $X$. It is easy to check that this module is uniserial 
of length four with subquotients of the radical filtration being 
$\hat{L}_2$, $\hat{L}_1$, $\hat{L}_2$, $\hat{L}_1$. Taking the 
restricted dual, we obtain a module which satisfies all the 
requirements of the lemma.
\end{proof}

\begin{lemma}\label{lem091}
Both $\hat{D}_{\lambda,\xi}$ and ${D}_{\lambda,\xi}$ are nontrivial
deformations when reduced modulo $\mathfrak{m}^2$.
\end{lemma}

\begin{proof}
As $\mathbb{C}[[x]]$ is a quotient of $\mathbb{C}[[x_1,x_2,\dots,x_n]]$, 
it is enough to prove the statement for the
smaller algebra ${D}_{\lambda,\xi}$. In case $n=2$ the claim 
follows directly from the well-known description of ${D}_{\lambda,\xi}$ 
(see \cite[Chapter~3]{Maz2}). Hence we may assume $n>2$. One could
observe that the claim follows from the main result of \cite{GS}
and our Theorem~\ref{thm29}. However, we would like to give an
independent argument based on our description
of ${D}_{\lambda,\xi}$ as a deformation, given by Theorem~\ref{thm21}.

Consider the element $x=\beta_1\alpha_1\in {D}_{\lambda_1,\xi_1}$.
We know that the image of $x$ in $A^{n-1}$ (that is, when we further
reduce modulo $\mathfrak{m}$) is a nonzero element $\overline{x}$
satisfying $\overline{x}^2=0$. Hence, to prove the statement of the 
lemma it is enough to show that $x^2$ is nonzero when reduced modulo
$\mathfrak{m}^2$.  For this it is enough to show that
$x^2$ does not annihilate some  module from 
${\mathcal{C}}^{(2)}_{\lambda,\xi}$.
This is a straightforward computation, which we describe below.

Set $R:=\mathcal{C}_{\lambda_1,\xi_1}\cong\mathbb{C}[[x]]$.
The projective module $R/\mathfrak{m}^2$ 
has length two with simple top and simple socle isomorphic to $L_1$. 
In particular, it is also injective in $\mathcal{C}^{(2)}_{\lambda_1,\xi_1}$. 
Applying the exact selfadjoint functor $\mathrm{F}_1$ to the short 
exact sequence
\begin{displaymath}
0\to L_1\to R/\mathfrak{m}^2\to L_1\to 0 
\end{displaymath}
we get the exact sequence
\begin{displaymath}
0\to M_1\to P_1(2)\to M_2\to 0, 
\end{displaymath}
where $P_1(2)$ is the projective cover of $\hat{L}_1$ and
$M_1\cong M_2\cong \mathrm{F}_1 L_1$. The module $\mathrm{F}_1 L_1$
is uniserial of length three, with simple top and simple socle isomorphic 
to $\hat{L}_1$ and the intermediate layer isomorphic to $\hat{L}_2$.
We assert that $x^2$ does not annihilate $P_1(2)$.

Similarly to the above one shows that for the projective-injective 
module $P_2(2)$ there is an exact sequence
\begin{displaymath}
0\to N_1\to P_2(2)\to N_2\to 0, 
\end{displaymath}
where $N_1\cong N_2$ has Loewy length three, with simple top and simple
socle isomorphic to $\hat{L}_2$, and the middle layer of the radical
(=socle) filtration isomorphic to $\hat{L}_1\oplus \hat{L}_3$
(the summand $\hat{L}_3$ should be omitted in case $n=3$).

We want to show that the composition
\begin{equation}\label{eq095}
P_1(2)\overset{\alpha_1}{\to}P_2(2)\overset{\beta_1}{\to}
P_1(2)\overset{\alpha_1}{\to}P_2(2)\overset{\beta_1}{\to}P_1(2)
\end{equation}
is nonzero. Note that the kernels of the morphisms $\alpha_1$ and 
$\beta_1$ contain only simple subquotients of the form $\hat{L}_1$ 
and $\hat{L}_2$, respectively (since we know that this is the case when 
we reduce modulo $\mathfrak{m}$).

Since $P_1(2)$ has a filtration with subquotients $M_1$ and $M_2$,
it has at most one submodule $N$ of length two,
both composition subquotients of which are isomorphic to $\hat{L}_1$.
At the same time the module $M$ from Lemma~\ref{lem092} must be
a quotient of $P_1(2)$. Therefore $N$ as above exists, is unique,
and $\alpha_1P_1(2)\cong P_1(2)/N\cong X$ is uniserial with 
layers of the radical filtration as described in Lemma~\ref{lem092}.

As the kernel of $\beta_1$ contains only 
simple subquotients of the form $\hat{L}_2$, it follows that 
$Y=\beta_1 X$ is uniserial of length three with the layers of the
radical filtration being $\hat{L}_1$, $\hat{L}_2$, $\hat{L}_1$.
As the kernel of $\alpha_1$ contains only 
simple subquotients of the form $\hat{L}_1$, it follows that 
$Z=\alpha_1 Y$ is uniserial of length two with the layers of the
radical filtration being $\hat{L}_1$, $\hat{L}_2$. 
As the kernel of $\beta_1$ contains only 
simple subquotients of the form $\hat{L}_2$, it follows that 
$\beta_1 Z\cong \hat{L}_1$. Therefore
$\alpha_1\beta_1\alpha_1\beta_1\neq 0$ and the claim of the lemma follows.
\end{proof}

\begin{proof}[Proof of Theorem~\ref{thmmain}\eqref{thmmain.3}
and Theorem~\ref{thmmain}\eqref{thmmain.4}.]
By  Theorem~\ref{thm21}, the algebra $D_{\lambda,\xi}$ is a deformation
of $A^{n-1}$ over $\mathbb{C}[[x]]$. By Lemma~\ref{lem091},
this deformation is nontrivial when reduced modulo $\mathfrak{m}^2$.
By Theorem~\ref{thm29}, the algebra described in 
Theorem~\ref{thmmain}\eqref{thmmain.3} (resp. \eqref{thmmain.4}) 
is also a deformation of $A^{n-1}$ over $\mathbb{C}[[x]]$
(resp. $\mathbb{C}[[x_1,x_2,\dots,x_n]]$), nontrivial 
when reduced modulo $\mathfrak{m}^2$. Hence the claim of 
Theorem~\ref{thmmain}\eqref{thmmain.3} (resp. \eqref{thmmain.4}) 
follows from Corollary~\ref{cordef}.
\end{proof}

\vspace{0.5cm}

\noindent
V.M.: Department of Mathematics, Uppsala University, SE 471 06,
Uppsala, SWEDEN, e-mail: {\tt mazor\symbol{64}math.uu.se}
\vspace{0.2cm}

\noindent
C.S.: Mathematik Zentrum, Universit{\"a}t Bonn,
Endenicher Allee 60, D-53115, Bonn, GERMANY,
e-mail: {\tt stroppel\symbol{64}uni-bonn.de}
\vspace{0.2cm}

\end{document}